\documentclass[12pt]{elsarticle}

\usepackage{amsmath} 

\usepackage{amsfonts}   
\usepackage{amssymb}

\usepackage{amsthm}

\usepackage{tikz,amsmath}
\usetikzlibrary{arrows, positioning}
   
\usepackage{amstext}      
   \theoremstyle{plain}%default
   \newtheorem{thm}{Theorem}[section]
   \newtheorem{prop}[thm]{Proposition}
   \newtheorem{lemma}[thm]{Lemma}  
   \newtheorem{cor}[thm]{Corollary}
   \theoremstyle{defn}
    \newtheorem{defn}[thm]{Definition}
   \theoremstyle{remark}
   
   \newtheorem{remark}[thm]{Remark}
   \newtheorem{example}[thm]{Example}

\newtheorem{ack}[thm]{Acknowledgement}

   \numberwithin{equation}{section}

\author{Johannes Christensen}

\address{Institut for Matematik, Aarhus University,  Denmark}

\title{KMS states on the Toeplitz algebras of higher-rank graphs}

\begin{document}

\begin{abstract}
The Toeplitz algebra $\mathcal{T}C^{*}(\Lambda)$ for a finite $k$-graph $\Lambda$ is equipped with a continuous one-parameter group $\alpha^{r}$ for each $ r\in \mathbb{R}^{k}$, obtained by composing the map $\mathbb{R} \ni t \to (e^{itr_{1}}, \dots , e^{itr_{k}}) \in \mathbb{T}^{k}$ with the gauge action on $\mathcal{T}C^{*}(\Lambda)$.  In this paper we give a complete description of the $\beta$-KMS states for the $C^{*}$-dynamical system $(\mathcal{T}C^{*}(\Lambda), \alpha^{r})$ for all finite $k$-graphs $\Lambda$ and all values of $\beta \in \mathbb{R}$ and $r\in \mathbb{R}^{k}$. 
\end{abstract}

\maketitle

\section{Introduction}
%Introduction to the subject
The structure of KMS-states for the gauge action or a generalised gauge action on the $C^{*}$-algebra of a directed graph has revealed a complex structure, even for finite graphs \cite{AA1, CT}. The same is true for the Toeplitz $C^{*}$-algebra of a finite directed graph \cite{AA1, AA2}, and it is therefore natural to seek for similar results for the canonical actions on the $C^{*}$-algebra and the Toeplitz $C^{*}$-algebra of a finite higher-rank graph \cite{KP}.

%review of known results
For the Toeplitz algebra of a finite strongly connected $k$-graph $\Lambda$ without sources and sinks the simplex of KMS states was described in \cite{aHLRS} for a specific dynamics defined using the vertex matrices of $\Lambda$ (this dynamics is called "preferred" in \cite{aHLRS}). Since then there has been contributions from a handful of papers where the objective has been to describe the KMS states for more general graphs and more general continuous one-parameter groups. The most recent contribution is \cite{FaHR}, where the authors describe an algorithm for determining the $\beta$-KMS simplex on the Toeplitz algebra of a finite $k$-graph $\Lambda$ and for a continuous one-parameter group defined using a vector $r\in \mathbb{R}^{k}$ subject to the conditions:
\begin{enumerate}
\item $\Lambda$ has no sinks and no sources.
\item $\beta >0$, $r \in (0, \infty)^{k}$ and $r$ has rationally independent coordinates.
\item There are no trivial strongly connected components and no isolated subgraphs in $\Lambda$.
\item For all components $C$ in $\Lambda$ the graph restricted to $C$, $\Lambda_{C}$, is coordinatewise irreducible and each vertex matrix for $\Lambda$ restricted to a component $C$ has spectral radius greater than $1$.
\item If two components $C$ and $D$ are connected by an edge of any color in the skeleton of $\Lambda$, then they are connected by edges of all $k$ colors.
\end{enumerate}

%Purpose of the article
The aim of this paper is to remove all of these conditions. We will describe the simplex of $\beta$-KMS states on the Toeplitz algebra $\mathcal{T}C^{*}(\Lambda)$ for the continuous one-parameter group $\alpha^{r}$ for all values of $\beta\in \mathbb{R}$, $r\in \mathbb{R}^{k}$ and all finite $k$-graphs $\Lambda$. Our results reveal that some of the above restrictions imposed in \cite{FaHR} greatly reduce the size and complexity of the simplex of $\beta$-KMS states, for example they imply that the simplex is finite-dimensional, while our more general approach reveals the existence of KMS-simplexes with uncountably many extreme points.  Furthermore our description does not involve any repeated algorithm, and we believe that this makes it much easier to use in concrete calculations.

%Outline
To describe the KMS states on $\mathcal{T}C^{*}(\Lambda)$ for a finite $k$-graph $\Lambda$ we proceed as follows: In section \ref{background} we present the theory on higher-rank graphs, groupoids and $C^{*}$-dynamical systems that we will need in the paper. Section \ref{linearalgebra} is devoted to a linear algebraic result concerning vectors that are sub-invariant under a family of commuting matrices. In section \ref{gaugeinvariant} we use the general result from section \ref{linearalgebra} to describe a bijection between certain vectors over $\Lambda^{0}$ and gauge-invariant KMS states on the Toeplitz algebra. We then proceed in section \ref{decompofKMS} to describe a decomposition of the gauge-invariant KMS states, which in section \ref{nongauge} allows us to use the theory developed in \cite{C} to describe all KMS states. To illustrate our results we use section \ref{ex} to present a few examples and compare our results with the literature.

%Techniques as in C
The techniques and approach in this paper are similar to the ones used in \cite{C} to describe the KMS states on the Cuntz-Kriger algebras of finite higher-rank graphs without sources, and especially the analysis in section \ref{nongauge} that describes the non gauge-invariant KMS states is heavily inspired by ideas in \cite{C}. The description of the gauge-invariant KMS states uses many ideas and techniques already described in the literature on the subject (e.g. in \cite{aHKR}, \cite{aHLRS1} and \cite{aHLRS}). We do however find the new insight obtained regarding gauge-invariant KMS states both interesting and non-trivial, and we consider this the main contribution of this paper.

\section{Background}\label{background}

\subsection*{Higher-rank graphs and their Toeplitz algebra}
We will in the following summarise our notation and conventions on higher-rank graphs. For an in-depth treatment we refer the reader to \cite{R, KP}. Throughout $\mathbb{N}$ denotes the natural numbers including zero. For $k \in \mathbb{N}$ with $k\geq 1$ we write $\{e_{1}, \dots , e_{k}\}$ for the standard generators for $\mathbb{N}^{k}$ and for $n,m \in \mathbb{N}^{k}$ we write $n \vee m$ for the pointwise maximum of $n$ and $m$. A higher-rank graph $(\Lambda, d)$ of rank $k \in \mathbb{N}$ with $k\geq 1$ is a pair consisting of a countable small category $\Lambda$ and a functor $d:\Lambda \to \mathbb{N}^{k}$ that has the factorisation property, i.e. if $d(\lambda)=n+m$ for some $\lambda \in \Lambda$ and $n,m \in \mathbb{N}^{k}$, then there exists unique $\mu,\eta \in \Lambda$ with $d(\mu)=n$, $d(\eta)=m$ and $\lambda=\mu \eta$. We define $\Lambda^{n}:=d^{-1}(\{n\})$ for each $n\in \mathbb{N}^{k}$. The factorisation property guarantees that we can identify the objects of the category $\Lambda$ with $\Lambda^{0}$, and we call them vertexes. Likewise we think of elements $\lambda$ of $\Lambda$ as paths in a graph with degree $d(\lambda)$, and we use the range and the source maps $r,s : \Lambda \to \Lambda^{0}$ to make sense of the start $s(\lambda)$ and the end $r(\lambda)$ of our path. Some times we will write $\Lambda$ instead of $(\Lambda, d)$ and simply call it a $k$-graph, in which case it is implicit that $k\geq 1$. Whenever $X,Y \subseteq \Lambda$ we let $XY$ denote the set of composed paths, and we use the usual conventions for defining sets of paths, e.g. $v\Lambda w :=\{w\}\Lambda\{w\}$ for $v,w \in \Lambda^{0}$. For $I\subseteq \{1, \dots , k \}$ we set:
\begin{equation*}
\Lambda_{I}:=\{ \lambda \in \Lambda \ : \ d(\lambda)_{j}=0 \text{ for all } j \in \{1, \dots , k \} \setminus I \} 
\end{equation*}
When $I \neq \emptyset$ $\Lambda_{I}$ can then be considered as a $\lvert I \rvert$-graph by defining a $d':\Lambda_{I} \to \mathbb{N}^{\lvert I \rvert}$ in the obvious way, but to make the notation more fluid we will let the degree functor be the restriction of $d$ to $\Lambda_{I}$, i.e. we identify $\mathbb{N}^{\lvert I \rvert}$ with 
\begin{equation*}
\mathbb{N}^{I}:=\{ n \in \mathbb{N}^{k} \ : \ n_{j}=0 \text{ for } j \notin I\}
\end{equation*}
Keeping in line with this notation, we will identify $\mathbb{N}^{k}$ with $\mathbb{N}^{I}\oplus \mathbb{N}^{J}$ whenever we have a partition $I\sqcup J=\{1, \dots , k\}$, and write $d(x)=(d(x)_{I}, d(x)_{J})$. For each subset $I \subseteq \{1, \dots , k \}$ we can define a relation $\leq_{I}$ on $\Lambda^{0}$ by letting $v\leq_{I} w$ if $v\Lambda_{I} w :=\{w\}\Lambda_{I}\{w\}\neq \emptyset$, and we can then define an equivalence relation $\sim_{I}$ on $\Lambda^{0}$ by defining $v \sim_{I} w$ when $v\leq_{I} w$ and $w\leq_{I} v$. We write $\sim$ instead of $\sim_{\{1, \dots , k\}}$, and when there can be no confusion about which relation $\sim_{I}$ we refer to we call the equivalence classes  \emph{components}. When a graph only has one component in $\sim$ we call it \emph{strongly connected}. Our $k$-graph $\Lambda$ is \emph{finite} when $\Lambda^{n}$ is a finite set for each $n \in \mathbb{N}^{k}$, and without sources when for each $v \in \Lambda^{0}$ and $n\in \mathbb{N}^{k}$ there is a $\lambda\in \Lambda^{n}$ with $r(\lambda)=v$, i.e $v\Lambda^{n} \neq \emptyset$. If $\Lambda$ is a finite $k$-graph, then $\Lambda_{I}$ is a finite $\lvert I \rvert$-graph for each $I\subseteq \{1, \dots , k \}$ with $I \neq \emptyset $. For $I \subseteq \{1, \dots , k\}$ and $V \subseteq \Lambda^{0}$ we define the \emph{$I$-closure} $\overline{V}^{I}$ of $V$ and the \emph{hereditary $I$-closure} $\widehat{V}^{I}$ as:
\begin{equation*}
\overline{V}^{I}:= \{ w\in \Lambda^{0} \ :\exists v \in V, \ w \leq_{I}v \} \quad , \quad 
\widehat{V}^{I}:= \{ w\in \Lambda^{0} \ :\exists v \in V, \ v \leq_{I}w \}
\end{equation*}
We write $\overline{V}:=\overline{V}^{\{1, \dots , k\}}$ and $\widehat{V}:=\widehat{V}^{\{1, \dots , k\}}$ and call it the closure and the hereditary closure of $V$. Letting $M_{S}(\mathbb{F})$ be the set of matrices over the finite set $S$ with entries in $\mathbb{F}$, the \emph{vertex matrices} $A_{1}, \dots , A_{k} \in M_{\Lambda^{0}}(\mathbb{N})$ for a finite $k$-graph $\Lambda$ are the matrices with entries $A_{i}(v,w)=\lvert v\Lambda^{e_{i}}w\rvert$. They commute pairwise, and setting $A^{n}:=\prod_{i=1}^{k} A_{i}^{n_{i}}$ for $n\in \mathbb{N}^{k}$ it follows that $A^{n}(v,w)=\lvert v \Lambda^{n}w\rvert$.

For a finite $k$-graph $\Lambda$, a \emph{Toeplitz-Cuntz-Krieger $\Lambda$-family} consists of partial isometries $\{S_{\lambda}  : \ \lambda \in \Lambda\} $ subject to the conditions:
\begin{enumerate}
\item $\{p_{v}:=S_{v} \ : v\in \Lambda^{0}\}$ are mutually orthogonal projections.
\item When $\lambda , \mu \in \Lambda$ with $s(\lambda)=r(\mu)$ we have $S_{\lambda \mu}=S_{\lambda}S_{\mu}$.
\item $S_{\lambda}^{*} S_{\lambda} =p_{s(\lambda)}$ for every $\lambda \in \Lambda$.
\item  $p_{v} \geq \sum_{\lambda \in v\Lambda^{n}}S_{\lambda}S_{\lambda}^{*}$ for all $v \in \Lambda^{0}$ and $n \in \mathbb{N}^{k}$.
\item $S_{\mu}^{*}S_{\lambda}=\sum_{(\kappa, \eta) \in \Lambda^{\text{min}} (\mu, \lambda)} S_{\kappa}S_{\eta}^{*}$ for all $\mu , \lambda \in \Lambda$.
\end{enumerate}
where $\Lambda^{\text{min}} (\mu, \lambda):=\{ (\kappa, \eta)\in \Lambda \times \Lambda \ : \mu \kappa = \lambda \eta, \ d(\mu \kappa)=d(\mu)\vee d(\lambda) \}$, see e.g. \cite{Cynthia, AA4, RS}. The Toeplitz algebra $\mathcal{T}C^{*}(\Lambda)$ of $\Lambda$ is then the $C^{*}$-algebra generated by a universal Toeplitz-Cuntz-Krieger $\Lambda$-family. It follows from the definition of $\mathcal{T}C^{*}(\Lambda)$ that $\mathcal{T}C^{*}(\Lambda)=\overline{\text{span}}\{S_{\lambda}S_{\mu}^{*}\ : \ \lambda, \mu \in \Lambda\}$ and that we have a strongly continuous action $\gamma: \mathbb{T}^{k} \to \text{Aut}(\mathcal{T}C^{*}(\Lambda))$ with:
\begin{equation*}
\gamma_{z}(S_{\lambda})=z^{d(\lambda)} S_{\lambda} \qquad \text{ for all } z \in \mathbb{T}^{k} \text{ and } \lambda \in \Lambda
\end{equation*}
where $z^{d(\lambda)}:=\prod_{i=1}^{k} z_{i}^{d(\lambda)_{i}}$.

\subsection*{$C^{*}$-dynamical systems and KMS states}
In this paper a $C^{*}$-dynamical system is a pair $(\mathcal{A}, \alpha)$ consisting of a $C^{*}$-algebra $\mathcal{A}$ and a continuous one-parameter group $\alpha$, i.e. a strongly continuous representation of $\mathbb{R}$ in $\text{Aut}(\mathcal{A})$. An element $a\in \mathcal{A}$ is analytic for $\alpha$ when there is an analytic extension of the map $\mathbb{R}\ni t \to \alpha_{t}(a)\in \mathcal{A}$ to the entire complex plane $\mathbb{C}$, and we then denote the value of this map at $z \in \mathbb{C}$ as $\alpha_{z}(a)$. A $\beta$-KMS state for the $C^{*}$-dynamical system $(\mathcal{A}, \alpha)$ is a state $\omega$ on $\mathcal{A}$ satisfying:
\begin{equation*} 
\omega(xy)=\omega(y \alpha_{i \beta}(x))
\end{equation*}
for all elements $x,y$ in a norm dense, $\alpha$-invariant $*$-algebra of $\mathcal{A}$ consisting of analytic elements for $\alpha$.

For any $r\in \mathbb{R}^{k}$ we can compose the map $\mathbb{R}\ni t \to e^{itr}:=(e^{itr_{j}})_{j=1}^{k} \in\mathbb{T}^{k}$ with the gauge action $\gamma$ on $\mathcal{T}C^{*}(\Lambda)$ to obtain a continuous one-parameter group $\alpha^{r}$. For all $\lambda, \mu \in \Lambda$ the map:
\begin{equation*}
\mathbb{R} \ni t\to\alpha^{r}_{t}(S_{\lambda} S_{\mu}^{*})=e^{it r\cdot (d(\lambda)-d(\mu))}S_{\lambda} S_{\mu}^{*}
\end{equation*} 
has an analytic extension to $\mathbb{C}$, and hence $S_{\lambda} S_{\mu}^{*}$ is an analytic element for $(\mathcal{T}C^{*}(\Lambda), \alpha^{r})$.

\subsection*{Realising $\mathcal{T}C^{*}(\Lambda)$ as a groupoid $C^{*}$-algebra} 
We follow \cite{aHKR} when introducing the groupoid of the Toeplitz algebra, and for a more rigorous treatment we refer the reader to \cite{Cynthia, yeend}. We write  $n \leq m$ for elements $n,m \in (\mathbb{R}\cup \{\infty\})^{k}$ when $n_{i}\leq m_{i}$ for all $i \in \{1, \dots , k\}$ and $n\lneq m$ when $n\leq m$ and $n\neq m$, and we use the same notation for the relation restricted to the subsets $(\mathbb{N}\cup \{\infty\})^{k}$ and $\mathbb{N}^{k}$. $\Omega_{k}=\{ (p,q)\in \mathbb{N}^{k} \times \mathbb{N}^{k} \ : p\leq q\} $ is the standard example of a $k$-graph $\Lambda$ without sources, and for $n\in (\mathbb{N}\cup \{\infty\})^{k}$ we set $\Omega_{k,n}$ equal to the subgraph $\{ (p,q)\in \Omega_{k} \ : q\leq n\}$. For a finite $k$-graph $\Lambda$ and each $n\in (\mathbb{N}\cup \{\infty\})^{k}$ we let $\Lambda^{n}$ denote the set of degree preserving functors $x:\Omega_{k,n} \to \Lambda$ and set $d(x):=n$ and $r(x):=x(0,0)$. When $n$ has finite entries this set can be identified with $d^{-1}(\{n\})$, so the notation does not collide with the one already introduced. Let:
\begin{equation*}
\Lambda^{*}:= \bigcup_{n \in (\mathbb{N}\cup \{\infty \})^{k}} \Lambda^{n}
\end{equation*}
and define for each $\lambda \in \Lambda$ the cylinder set $Z(\lambda):=\{x\in \Lambda^{*} \ : x(0, d(\lambda))=\lambda \}$. For each finite set $F\subseteq s(\lambda)\Lambda$ set 
\begin{equation*}
Z(\lambda \setminus F):=Z(\lambda)\setminus\left( \bigcup_{\mu \in F}Z(\lambda \mu) \right).
\end{equation*}
The sets $Z(\lambda \setminus F)$ then form a basis of compact open sets for a second countable locally compact Hausdorff topology on $\Lambda^{*}$, and since $\Lambda^{*}=\bigcup_{v\in \Lambda^{0}}Z(v)$  it follows that $\Lambda^{*}$ is compact. Whenever we have a partition $I\sqcup J=\{1, \dots , k\}$ we write $\infty_{I}:=(\infty)_{i\in I}\in (\mathbb{N}\cup\{\infty\})^{I}$ for the element with $(\infty_{I})_{i}=\infty$ for all $i\in \mathbb{I}$ and for $m\in \mathbb{N}^{J}$ we set:
\begin{equation*}
\Lambda^{\infty_{I}, m} :=\{x\in \Lambda^{*} \ : \ d(x)=(\infty_{I}, m) \}
\end{equation*}
which is a Borel set by Proposition 3.2 in \cite{aHKR}, and we set $\partial^{I}\Lambda:=\bigcup_{m\in \mathbb{N}^{J}}\Lambda^{\infty_{I}, m}$. When $I =\emptyset$ then $\Lambda^{\infty_{I}, m}=\Lambda^{m}$ and $\partial^{I}\Lambda=\Lambda$. When $I=\{1, \dots , k\}$ then $\Lambda^{\infty}:=\Lambda^{\infty_{I}, 0}$ is the \emph{infinite path space} of $\Lambda$. It follows that we have a Borel partition of $\Lambda^{*}$, i.e.:
\begin{equation*}
\Lambda^{*}=\bigsqcup_{I \subseteq \{1, \dots , k\}} \partial^{I}\Lambda
\end{equation*}
For each $n \in \mathbb{N}^{k}$ the formula $\sigma^{n}(x)(p,q)=x(p+n, q+n)$ defines a map $\sigma^{n}$ on $\{ x\in \Lambda^{*}\ : d(x)\geq n\}$ which we call the \emph{shift map}. We can then define a groupoid $\mathcal{G}_{\Lambda}$ as:
\begin{equation*}
\mathcal{G}_{\Lambda}: = \{ (x, p-q,y) \in \Lambda^{*} \times \mathbb{Z}^{k} \times \Lambda^{*} \ : \ p\leq d(x), \ q\leq d(y) , \ \sigma^{p}(x) = \sigma^{q}(y) \}
\end{equation*} 
with the usual composition and inverse. We equip $\mathcal{G}_{\Lambda}$ with a topology such that it becomes a locally compact second countable Hausdorff \'etale groupoid, satisfying that the full groupoid $C^{*}$-algebra $C^{*}(\mathcal{G}_{\Lambda})$ is isomorphic to $\mathcal{T}C^{*}(\Lambda)$, that the unitspace $\mathcal{G}_{\Lambda}^{(0)}$ is isomorphic to $\Lambda^{*}$ with the topology generated by the sets $Z(\lambda \setminus F)$, and that $C(\Lambda^{*})\simeq \overline{\text{span}}\{ S_{\lambda}S_{\lambda}^{*} \ : \lambda \in \Lambda \}$ under an isomorphism that maps $1_{Z(\lambda)}\to S_{\lambda}S_{\lambda}^{*}$ for each $\lambda \in \Lambda$. Furthermore the continuous one-parameter group $\alpha^{r}$ corresponds to the one arising from the groupoid homomorphism $c_{r}(x,n,y)=r\cdot n$ in the groupoid picture of $\mathcal{T}C^{*}(\Lambda)$, and the topology on $\mathcal{G}_{\Lambda}$ makes the map $\Phi: \mathcal{G}_{\Lambda} \to \mathbb{Z}^{k}$ given by $\Phi(x,n,y):=n$ continuous. Since the definition of the topology on $\mathcal{G}_{\Lambda}$ is not crucial for our exposition, we refer the reader to Appendix B in \cite{aHKR} or \cite{yeend} for the details.

When considering the groupoid picture of $\mathcal{T}C^{*}(\Lambda)$ every $\beta$-KMS state $\omega$ on $\mathcal{T}C^{*}(\Lambda)$ for $\alpha^{r}$ gives rise to a Borel probability measure $m$ on $\Lambda^{*}$ by using the Riesz Representatiom Theorem on $\omega$ restricted to $C(\Lambda^{*})$. We say that this measure is \emph{the measure associated to $\omega$}, and by Theorem 1.3 in \cite{N} such measure are exactly the probability measures that are quasi-invariant with Radon-Nikodym cocycle $e^{-\beta c_{r}}$. Since any such measure restricted to an invariant Borel subset of $\Lambda^{*}$, i.e. a Borel set $B$ with $s(r^{-1}(B))=B$, is again a quasi-invariant measure with Radon-Nikodym cocycle $e^{-\beta c_{r}}$, it follows that the extremal quasi-invariant probability measure with Radon-Nikodym cocycle $e^{-\beta c_{r}}$ maps invariant sets of $\Lambda^{*}$ into $\{0,1\}$.

\section{Decomposition of sub-invariant vectors} \label{linearalgebra}
To decompose our KMS states it is necessary to decompose certain vectors over the set of vertexes, and since our solution to this problem is purely linear algebraic and works for very general sets and vectors, we have devoted this section to present it in its full generality. Regarding notation $\mathbb{R}_{+}=\{r\in \mathbb{R} \ : \ r\geq 0\}$ and we write $\prod_{i=1}^{l}H_{i} x$ for matrices $H_{i}$ and an expression $x$ to mean $H_{1}\cdots H_{l} x$. For any finite set $S$ we let $1_{S}\in M_{S}(\mathbb{R})$ denote the identity matrix.

\begin{defn}
Let $S$ be a finite set and $B_{1}, \dots , B_{k} \in M_{S}(\mathbb{R}_{+})$ be pairwise commuting, i.e. $B_{i}B_{j}=B_{j}B_{i}$ for all $i,j$. We say a vector $\psi\in [0, \infty[^{S}$ is sub-invariant for the family $\{B_{i} \}_{i=1}^{k}$ if:
\begin{equation} \label{eaiv}
\prod_{i \in I} (1_{S}-B_{i}) \psi \geq 0 \qquad \text{ for each subset } I \subseteq \{1, \dots , k\}.
\end{equation} 
\end{defn}

\begin{prop} \label{p32}
Let $S$ be a finite set, $B_{1}, \dots , B_{k} \in M_{S}(\mathbb{R}_{+})$ be pairwise commuting and $\psi$ be a sub-invariant vector for the family $\{B_{i}\}_{i=1}^{k}$. For each subset $I \subseteq \{1,2, \dots , k\}$ there exists a vector $h^{I}$ that is sub-invariant for the family $\{B_{i}\}_{i=1}^{k}$ such that:
\begin{enumerate}
\item\label{inv1} $B_{i}h^{I}=h^{I}$ for all $i \in I$.
\item \label{inv2}$\lim_{n \to \infty} B_{j}^{n}h^{I} =0$ for $j \in \{1, \dots , k\}\setminus I$.
\item \label{inv3}$\psi = \sum_{I \subseteq \{1,\dots , k\}} h^{I}$.
\end{enumerate}
Furthermore this decomposition is unique in the sense that there is only one family of sub-invariant vectors satisfying \ref{inv1}-\ref{inv3}.
\end{prop}

\begin{proof}
Before proving the Proposition we will make an observation regarding the Riesz-decomposition of vectors, see e.g Theorem 5.6 in \cite{S}. Assume $\psi \in [0, \infty[^{S}$ satisfies $B_{1} \psi \leq \psi$ for some $B_{1} \in M_{S}(\mathbb{R}_{+})$. The Riesz decomposition then says that we can write $\psi=\psi_{1}+\psi_{2}$ with:
\begin{equation} \label{eriesz}
\psi_{1}:=\lim_{n \to \infty} B_{1}^{n}\psi \quad , \quad \psi_{2} :=\sum_{n=0}^{\infty} B_{1}^{n}(\psi -B_{1}\psi)
\end{equation}
If there exists a (possibly empty) family $B_{2}, \dots , B_{k} \in M_{S}(\mathbb{R}_{+})$ such that $\{B_{i} \}_{i=1}^{k}$ is a family of pairwise commuting matrices and $\psi$ is sub-invariant for $\{B_{i}\}_{i=1}^{k}$, then we claim that $\psi_{1}$ and $\psi_{2}$ are also sub-invariant for $\{B_{i}\}_{i=1}^{k}$. To prove this let $J \subseteq \{1,2, \dots , k\}$ be arbitrary. Since:
\begin{equation*}
\prod_{j\in J}(1_{S}-B_{j})  B_{1}^{n}\psi  =  B_{1}^{n} \prod_{j\in J}(1_{S}-B_{j}) \psi \geq 0
\end{equation*}
for each $n\in \mathbb{N}$ it follows that $\psi_{1}$ is sub-invariant. If $1\notin J$, then:
\begin{equation*}
\prod_{j\in J}(1_{S}-B_{j}) \sum_{n=0}^{\infty} B_{1}^{n}(\psi - B_{1}\psi)
 =  \sum_{n=0}^{\infty} B_{1}^{n}\prod_{j\in J\cup\{1\}}(1_{S}-B_{j})\psi .
\end{equation*}
Every vector in the sum is non-negative by assumption, so this is a non-negative vector. If $1 \in J$ then:
\begin{align*}
\prod_{j\in J}(1_{S}-B_{j}) \sum_{n=0}^{\infty} B_{1}^{n}(\psi - B_{1}\psi)
&= (1_{S}-B_{1}) \sum_{n=0}^{\infty} B_{1}^{n}\prod_{j \in J}(1_{S}- B_{j}) \psi \\
&= \prod_{j \in J}(1_{S}- B_{j}) \psi \geq 0 .
\end{align*}
Hence $\psi_{1}$ and $\psi_{2}$ are sub-invariant, finishing our observation regarding the Riesz-decomposition.

To prove that the family of vectors in the Proposition exists we will prove the stronger statement that when $B_{1}, \dots , B_{k} \in M_{S}(\mathbb{R}_{+})$ are pairwise commuting and $\psi$ is a sub-invariant vector for the family $\{B_{i}\}_{i=1}^{k}$, there exists a family of vectors $h^{I}$, $I\subseteq \{1,2, \dots , k\}$, that satisfy \ref{inv1}-\ref{inv3} and $(I)$:

\begin{itemize}
\item[(I)] \label{itemI} If $A_{1}, \dots , A_{p}\in M_{S}(\mathbb{R}_{+})$ is a family of pairwise commuting matrices with $p\geq k$ and $A_{i}=B_{i}$ for $i\leq k$ satisfying that $\psi$ is sub-invariant for $\{A_{i}\}_{i=1}^{p}$, then the vectors $h^{I}$,$I\subseteq \{1,2, \dots , k\}$, will also be sub-invariant for $\{A_{i}\}_{i=1}^{p}$.
\end{itemize}
We will prove this by induction over $k$.

If $k=1$ then $\psi$ is sub-invariant under $B_{1}$, and hence $B_{1} \psi \leq \psi$. Use the Riesz-decomposition to write $\psi=\psi_{1}+\psi_{2}$ with $\psi_{1}$ and $\psi_{2}$ as in \eqref{eriesz}. Clearly $B_{1} \psi_{1}=\psi_{1}$ and $\lim_{n \to \infty} B_{1}^{n} \psi_{2}=0$, so setting $h^{\{1\}}:=\psi_{1}$ and $h^{\emptyset}:=\psi_{2}$ we have constructed a family satisfying \ref{inv1}-\ref{inv3}. Our observation on the Riesz-decomposition implies that the family also satisfies $(I)$.

Assume now that the statement is true for a $k\in \mathbb{N}$. Let $B_{1}, \dots , B_{k}, B_{k+1} \in M_{S}(\mathbb{R}_{+})$ be pairwise commuting and let $\psi\in [0, \infty[^{S}$ be sub-invariant for the family $\{B_{i} \}_{i=1}^{k+1}$. In particular $\psi$ is sub-invariant for the family $\{B_{i} \}_{i=1}^{k}$, so our induction hypothesis implies there there exists a family $\tilde{h}^{I}$, $I\subseteq \{1, \dots , k\}$ satisfying \ref{inv1}- \ref{inv3} and condition $(I)$ for the family $\{B_{i} \}_{i=1}^{k}$. Condition $(I)$ implies that each $\tilde{h}^{I}$ is sub-invariant for $\{B_{i} \}_{i=1}^{k+1}$, so in particular $B_{k+1} \tilde{h}^{I} \leq \tilde{h}^{I}$, and we can use the Riesz-decomposition to write $\tilde{h}^{I}=h^{I\cup\{k+1\}}+h^{I}$ with:
\begin{equation*}
h^{I\cup\{k+1\}}:=\lim_{n \to \infty} B_{k+1}^{n}\tilde{h}^{I} \quad , \quad h^{I}:=\sum_{n=0}^{\infty} B_{k+1}^{n}(\tilde{h}^{I} -B_{k+1}\tilde{h}^{I})
\end{equation*}
where $B_{k+1}h^{I\cup\{k+1\}}=h^{I\cup\{k+1\}}$ and $\lim_{n\to \infty} B_{k+1}^{n}h^{I}=0$. For $i\in I$ we have:
\begin{equation*}
B_{i}h^{I\cup\{k+1\}} = \lim_{n \to \infty} B_{k+1}^{n}B_{i}\tilde{h}^{I}= h^{I\cup\{k+1\}}
\end{equation*}
and
\begin{equation*}
B_{i}h^{I}
 =\sum_{n=0}^{\infty} B_{k+1}^{n}(B_{i}\tilde{h}^{I} -B_{k+1} B_{i}\tilde{h}^{I})
= h^{I}.
\end{equation*}
For $j\in \{1, \dots , k+1\} \setminus (I\cup\{k+1\})$ the inequalities $0\leq B_{j}^{n} h^{I} \leq B_{j}^{n} \tilde{h}^{I}$ and $0\leq B_{j}^{n} h^{I\cup\{k+1\}} \leq B_{j}^{n} \tilde{h}^{I}$ imply that 
\begin{equation*}
\lim_{n\to \infty} B_{j}^{n} h^{I}=\lim_{n\to \infty} B_{j}^{n} h^{I\cup\{k+1\}}=0 
\end{equation*}
so our new family $h^{I}$, $I \subseteq \{1, \dots , k+1\}$, satisfies \ref{inv1}-\ref{inv3} for $\{B_{i}\}_{i=1}^{k+1}$. Assume $A_{1}, \dots , A_{p}\in M_{S}(\mathbb{R}_{+})$ is a family of pairwise commuting matrices with $p\geq k+1$ and $A_{i}=B_{i}$ for $i\leq k+1$, and assume that $\psi$ is sub-invariant for $\{A_{i}\}_{i=1}^{p}$. By our induction hypothesis each $\tilde{h}^{I}$ is sub-invariant for $\{A_{i}\}_{i=1}^{p}$, and using our observation on the Riesz-decomposition it follows that each $h^{I}$, $I\subseteq \{1, \dots ,k+1 \}$, is sub-invariant for $\{A_{i}\}_{i=1}^{p}$ as well. The existence statement in the Proposition now follows by induction.

To prove that the decomposition is unique, assume that there exists two families $h^{I}$ and $\tilde{h}^{I}$, $I \subseteq \{1, \dots , k\}$, that are sub-invariant and satisfies \ref{inv1}-\ref{inv3} for the family $\{B_{i}\}_{i=1}^{k}$. It then follows that the expression:
\begin{equation*}
\lim_{n_{1}\to \infty} B_{1}^{n_{1}} \lim_{n_{2}\to \infty} B_{2}^{n_{2}} \cdots \lim_{n_{k}\to \infty} B_{k}^{n_{k}} \psi 
\end{equation*}
is equal to both $h^{\{1,\dots , k\}}$ and $\tilde{h}^{\{1,\dots , k\}}$. Assume now that $h^{I}=\tilde{h}^{I}$ for all subsets $I \subseteq \{1, \dots , k\}$ with $\lvert I \rvert \geq n$ for some $1 \leq n \leq k$, and take a set $J\subseteq \{1,\dots , k\}$ with $\lvert J \rvert=n-1$. We write $J=\{j_{1}, \dots , j_{n-1} \}$. Taking the limits:
\begin{equation*}
\lim_{m_{1}\to \infty} B_{j_{1}}^{m_{1}} \lim_{m_{2}\to \infty} B_{j_{2}}^{m_{2}} \cdots \lim_{m_{n-1}\to \infty} B_{j_{n-1}}^{m_{n-1}} \psi 
\end{equation*}
we get:
\begin{equation*}
\sum_{I \supseteq J} \tilde{h}^{I} = \sum_{I \supseteq J} h^{I} .
\end{equation*}
By assumption $\tilde{h}^{I}=h^{I}$ for all $I \neq J$ in this sum, so we must have that $\tilde{h}^{J}=h^{J}$. It now follows from induction that the family $h^{I}$ is unique.
\end{proof}

\section{A description of the gauge-invariant KMS states} \label{gaugeinvariant}
The first step in our analysis of the KMS states on the Toeplitz algebra of a finite $k$-graph $\Lambda$ is to describe the ones that are gauge-invariant. In this section we will reduce the problem of finding gauge-invariant KMS states to the much simpler problem of finding certain invariant vectors over $\Lambda^{0}$. We remind the reader that for a finite set $S$ the $1$-norm for a vector $\psi\in \mathbb{R}^{S}$ is given by $\lVert\psi \rVert_{1}=\sum_{s \in S}\lvert \psi_{s} \rvert$.

\begin{lemma}\label{l41}
Let $\Lambda$ be a finite $k$-graph and let $r \in \mathbb{R}^{k}$ and $\beta \in \mathbb{R}$. Let $\omega$ be a $\beta$-KMS state for $\alpha^{r}$ and set $\psi_{v}:=\omega(p_{v})$ for each $v \in \Lambda^{0}$. Then $\psi\in [0, \infty[^{\Lambda^{0}}$ is a sub-invariant vector for the family $\{e^{-\beta r_{i}}A_{i} \}_{i=1}^{k}$ of unit $1$-norm. 
\end{lemma}

\begin{proof}
When $\beta\geq 0$, $r\in ]0, \infty[^{k}$ and $\Lambda$ has no sources this statement is part $(a)$ of Proposition 4.1 in \cite{aHLRS1}. When interpreting empty sums as $0$ the proof given there works for general $\beta \in \mathbb{R}$, $r\in \mathbb{R}^{k}$ and finite $k$-graphs, so we will not give it here.
\end{proof}

Lemma \ref{l41} gives us an affine map from the set of gauge-invariant $\beta$-KMS states for $\alpha^{r}$ on $\mathcal{T}C^{*}(\Lambda)$ to the set of non-negative sub-invariant vectors for the family $\{e^{-\beta r_{i}}A_{i}\}_{i=1}^{k}$ of unit $1$-norm. Proposition \ref{p43} below implies that it is a bijection. To prove this we need the following description of $\Lambda^{\infty_{I}, m}$ when we have a partition $I\sqcup J = \{1, \dots , k\}$ with $I\neq \emptyset$. Set:
\begin{equation*}
\Lambda^{0}(I):=\{v\in \Lambda^{0} \ : \ v\Lambda^{(n,0)}\neq \emptyset \text{ for all } n\in \mathbb{N}^{I}\},
\end{equation*}
i.e. $\Lambda^{0}(I)$ are the vertexes that are not sources in $\Lambda_{I}$. For $(n,m) \in \mathbb{N}^{I}\oplus \mathbb{N}^{J}$ we set:
\begin{equation*}
\mathcal{U}_{I}^{(n,m)} =\{ \lambda\in \Lambda^{(n,m)} \ : \ s(\lambda(0,l))\in \Lambda^{0}(I) \text{ for each } 0\leq l \leq (n,m)    \} 
\end{equation*}
Giving $\mathcal{U}_{I}^{(n,m)}$ the discrete topology we can for each $n, l \in \mathbb{N}^{I}$ with $n \leq l$ define a continuous map $\pi_{l, n}: \mathcal{U}_{I}^{(l,m)} \to \mathcal{U}_{I}^{(n,m)}$ by $\pi_{l, n}(\lambda)=\lambda(0, (n,m))$.

\begin{lemma} \label{l1}
Assume $I \neq \emptyset$. Then $\Lambda^{0}(I)$ is closed in $\Lambda$, $\pi_{l,n}$ is surjective and $\lim_{\leftarrow n\in \mathbb{N}^{I}} \mathcal{U}_{I}^{(n,m)}$ is homeomorphic to 
$\Lambda^{\infty_{I}, m}$.
\end{lemma}
 
\begin{proof}
To see that $\Lambda^{0}(I)$ is closed, let $\lambda \in v \Lambda w $ with $ w\in \Lambda^{0}(I)$ and let $n\in \mathbb{N}^{I}$. For $\mu\in w\Lambda^{(n,0)}$ then $\lambda\mu \in v \Lambda^{(n,0)+d(\lambda)}$, and hence by the unique factorisation property there are paths $\lambda'\in v\Lambda^{(n,0)}$ and $\mu '\in \Lambda^{d(\lambda)}$ such that $\lambda \mu = \lambda '  \mu '$, so $v\Lambda^{(n,0)} \neq \emptyset$. It follows that $ \Lambda^{0}(I)$ is closed. When $\mu \in \mathcal{U}_{I}^{(n,m)}$ then $s(\mu) \in \Lambda^{0}(I)$, so for each $s\in \mathbb{N}$ we can choose $\lambda_{s} \in s(\mu)\Lambda$ with $d(\lambda_{s})=l-n+\sum_{i\in I} se_{i}$. Since $\Lambda^{l-n}$ is finite, there is a $\lambda \in s(\mu)\Lambda^{l-n}$ with $\lambda_{s}(0, l-n)=\lambda$ for infinitely many $s$, and it follows that $\mu\lambda\in \mathcal{U}_{I}^{(l,m)}$, proving that $\pi_{l,n}$ is surjective.

Denote by $\tilde{\pi}_{l,n}$ the map $\Lambda^{(l,m)}\to \Lambda^{(n,m)}$ given by $\tilde{\pi}_{l,n}(\lambda)=\lambda(0, (n,m))$, so $\pi$ is a restriction of $\tilde{\pi}$. The map:
\begin{equation} \label{ehom}
\lim_{\leftarrow n\in \mathbb{N}^{I}} \Lambda^{(n,m)}  \ni \{\lambda_{n}\}_{n\in \mathbb{N}^{I}} \to \{\lambda_{n}\}_{n\in \mathbb{N}^{I}} \in \lim_{\leftarrow n\in \mathbb{N}^{I}} \mathcal{U}_{I}^{(n,m)} 
\end{equation}
is well defined, because for each $n, n' \in \mathbb{N}^{I}$ the element $\lambda_{n} \in \Lambda^{(n,m)}$ satisfies that $\lambda_{n+n'}\in \Lambda^{(n+n',m)}$ can be decomposed $\lambda_{n+n'}=\lambda_{n} \mu$ with $\mu \in s(\lambda_{n})\Lambda^{(n',0)}$, so since $n'$ was arbitrary $s(\lambda_{n})\in \Lambda^{0}(I)$. Standard arguments imply that \eqref{ehom} is a continuous bijection, and so since $\lim_{\leftarrow n\in \mathbb{N}^{I}} \Lambda^{(n,m)}$ is compact it is also a homeomorphism. Since Proposition 3.2 in \cite{aHKR} implies that $\lim_{\leftarrow n\in \mathbb{N}^{I}} \Lambda^{(n,m)}$ is homeomorphic to $\Lambda^{\infty_{I}, m}$ this proves the Lemma.
\end{proof}

The construction of the KMS state in the proof of Proposition \ref{p43} has a predecessor in Theorem 5.1 in \cite{aHKR}. 

\begin{prop} \label{p43}
Let $\Lambda$ be a finite $k$-graph, $r\in \mathbb{R}^{k}$, $\beta \in \mathbb{R}$ and let $\psi \in [0, \infty[^{\Lambda^{0}}$ be a sub-invariant vector for the family $\{e^{-\beta r_{i}}A_{i}\}_{i=1}^{k}$ of unit $1$-norm. Then there exists a unique gauge-invariant $\beta$-KMS state $\omega_{\psi}$ for $\alpha^{r}$ on $\mathcal{T}C^{*}(\Lambda)$ such that $\omega_{\psi}(p_{v})=\psi_{v}$ for each $v\in \Lambda^{0}$.
\end{prop} 

\begin{proof}
Assume $\omega$ and $\omega'$ are gauge-invariant $\beta$-KMS states for $\alpha^{r}$ with $\omega(p_{v})=\omega'(p_{v})$ for all $v \in \Lambda^{0}$. Lemma 3.1 in \cite{C} implies that both $\omega$ and $\omega'$ are determined by their values on the elements $S_{\lambda} S_{\lambda}^{*}$, $\lambda \in \Lambda$. Since:
\begin{equation*}
\omega(S_{\lambda} S_{\lambda}^{*})
= e^{- \beta r \cdot d(\lambda)} \omega(p_{s(\lambda)} )
= e^{- \beta r \cdot d(\lambda)} \omega'(p_{s(\lambda)} )
=\omega'(S_{\lambda} S_{\lambda}^{*})
\end{equation*}
we must have $\omega = \omega'$, which proves that if the state $\omega_{\psi}$ exists it is unique. Proposition \ref{p32} implies that it is enough to prove that $\omega_{\psi}$ exists when there is a partition $I \sqcup J=\{1, 2, \dots , k\}$ with $e^{-\beta r_{i}}A_{i}\psi =  \psi$ for $i \in I$ and $\lim_{n \to \infty}(e^{-\beta r_{j}}A_{j})^{n} \psi=0$ for $j \in J$, so we assume this is the case. We now define a vector $\phi$ by:
\begin{equation*}
\phi:=\prod_{j\in J} (1_{\Lambda^{0}}-e^{-\beta r_{j}}A_{j})\psi.
\end{equation*}
When $J= \emptyset$ we interpret this as $\phi:=\psi$. Notice $\phi\in [0, \infty[^{\Lambda^{0}}$ since $\psi$ is sub-invariant, it is however not clear yet that $\phi \neq 0$. We will now define measures $\nu^{m}$ on $\Lambda^{\infty_{I}, m}$ for each $m\in\mathbb{N}^{J}$ using $\phi$. When $I=\emptyset$, we define $\nu^{m}$ on $\Lambda^{\infty_{I}, m}=\Lambda^{m}$ by $\nu^{m}(\{\lambda\})=e^{-\beta r \cdot m}\phi_{s(\lambda)}$. When $I\neq \emptyset$ give the finite set $\mathcal{U}_{I}^{(n,m)}$ the discrete topology for each $(n,m) \in \mathbb{N}^{I}\oplus \mathbb{N}^{J}$, and define a measure $\nu^{n,m}$ on $\mathcal{U}_{I}^{(n,m)}$ by:
\begin{equation}
\nu^{n,m}(\{ \lambda \})=e^{-\beta r\cdot (n,m)} \phi_{s(\lambda)} \qquad \text{ for } \lambda \in \mathcal{U}_{I}^{(n,m)}.
\end{equation}
Since the vertex matrices commute it follows from the definition of $\phi$ that $e^{-\beta r_{i}}A_{i} \phi=\phi$ for $i \in I$. For $v\in \Lambda^{0} \setminus \Lambda^{0}(I)$ there is a $n \in \mathbb{N}^{I}$ with $A^{(n,0)}(v,u)=0$ for all $u$, and hence $\phi_{v} = e^{-\beta r \cdot(n,0)}(A^{(n,0)} \phi)_{v} = 0$. Since $\Lambda^{0}(I)$ is closed by Lemma \ref{l1} we get for any $\lambda \in \mathcal{U}_{I}^{(n,m)}$:
\begin{align*}
&\nu^{l,m}(\pi_{l,n}^{-1}(\{\lambda \}))= \sum_{\mu \in \pi_{l,n}^{-1}(\{\lambda \})} e^{-\beta r\cdot (l,m)} \phi_{s(\mu)}
= \sum_{\eta \in s(\lambda)\mathcal{U}_{I}^{(l-n,0)}} e^{-\beta r\cdot (l,m)} \phi_{s(\eta)}\\
&= \sum_{w \in \Lambda^{0}}\sum_{\eta \in s(\lambda)\mathcal{U}_{I}^{(l-n,0)}w} e^{-\beta r\cdot (l,m)} \phi_{w}
=\sum_{w \in \Lambda^{0}}\sum_{\eta \in s(\lambda)\Lambda^{(l-n, 0)}w} e^{-\beta r\cdot (l,m)} \phi_{w} 
 \\
&= e^{-\beta r\cdot (l,m)} (A^{(l-n, 0)} \phi )_{s(\lambda)}
= e^{-\beta r\cdot (n,m)}  \phi_{s(\lambda)}
= \nu^{n,m}(\{\lambda\})
\end{align*}
Letting $\pi_{n} :\lim_{\leftarrow n\in \mathbb{N}^{I}}\mathcal{U}_{I}^{(n,m)} \to \mathcal{U}_{I}^{(n,m)}$ be the natural projection of the inverse limit for each $n\in \mathbb{N}^{I}$, a standard argument (using e.g. Lemma 5.2 in \cite{aHKR}) gives the existence of a Borel measure $\nu^{m}$ on $\lim_{\leftarrow n\in \mathbb{N}^{I}}\mathcal{U}_{I}^{(n,m)}$ satisfying $\nu^{m}(\pi_{n}^{-1}(\{\lambda\}))=\nu^{n,m}(\{\lambda\})$ for each $n\in \mathbb{N}^{I}$ and $\lambda \in \mathcal{U}_{I}^{(n,m)}$. We consider $\nu^{m}$ as a Borel measure on $\Lambda^{*}$ with $\nu^{m}(\Lambda^{\infty_{I}, m})=\nu^{m}(\Lambda^{*})$. By construction it satisfies:
\begin{equation} \label{eqnice}
\nu^{m}(Z(\lambda)\cap \Lambda^{\infty_{I}, m})=\nu^{n,m}(\{\lambda\}) = e^{- \beta r\cdot d(\lambda)} \phi_{s(\lambda)}
\end{equation}
for each $\lambda \in \mathcal{U}_{I}^{(n,m)}$. If $\lambda \in \Lambda^{(n,m)} \setminus \mathcal{U}_{I}^{(n,m)}$ then \eqref{eqnice} still holds true since both sides are $0$. The measure $\nu^{m}$ constructed when $I=\emptyset$ also satisfies \eqref{eqnice}. We will now construct a measure $\nu$ on $\partial^{I}\Lambda$ by summing all of the measures $\nu^{m}$, $m\in \mathbb{N}^{J}$. When $J=\emptyset$ we have only constructed a measure $\nu^{0}$ on $\Lambda^{\infty_{I}, 0}$, so we set $\nu=\nu^{0}$ and notice that by \eqref{eqnice} $\nu(Z(v))=\psi_{v}$ for each $v\in \Lambda^{0}$. When $J\neq \emptyset$ we can use \eqref{eqnice} for any $\mu \in \Lambda$ with $l:=d(\mu)_{J} \leq m$ to see that:
\begin{align} \label{ecirkel}
\nu^{m}(Z(\mu))&=\nu^{m}(Z(\mu)\cap \Lambda^{\infty_{I}, m})= \sum_{\lambda \in s(\mu)\Lambda^{(0,m-l)}} \nu^{m}(Z(\mu \lambda)\cap \Lambda^{\infty_{I}, m}) \\
\nonumber &= \sum_{\lambda \in s(\mu)\Lambda^{(0,m-l)}} e^{- \beta r \cdot (d(\mu)_{I},m)} \phi_{s(\lambda)}
= e^{- \beta r \cdot (d(\mu)_{I},m)} (A^{(0,m-l)}\phi)_{s(\mu)}
\end{align}
In particular, we have that $\nu^{m}(Z(v))=e^{- \beta r \cdot (0,m)} (A^{(0,m)}\phi)_{v}$ for all $v\in \Lambda^{0}$. Taking a $M \in \mathbb{N}^{J}$ we see that:
\begin{align*}
&\sum_{0\leq m \leq M}e^{-\beta r \cdot (0,m)} A^{(0,m)}\phi
= \sum_{0\leq m \leq M}\prod_{j \in J}\left( e^{-\beta r_{j}} A_{j} \right)^{m_{j}}\phi
=  \prod_{j \in J} \sum_{m_{j}=0}^{M_{j}}  \left(e^{-\beta r_{j}} A_{j} \right)^{m_{j}}\phi \\
&= \left[\prod_{j \in J} \sum_{m_{j}=0}^{M_{j}}  \left(e^{-\beta r_{j}} A_{j} \right)^{m_{j}}(1_{\Lambda^{0}}-e^{-\beta r_{j}}A_{j})\right]\psi
= \left[ \prod_{j \in J} (1_{\Lambda^{0}}-(e^{-\beta r_{j}} A_{j})^{M_{j}+1})\right]\psi \\
&=\left[ \sum_{L \subseteq J} (-1)^{\lvert L \rvert} \prod_{j \in L} (e^{-\beta r_{j}} A_{j})^{M_{j}+1}\right]\psi 
\end{align*}
By choice of $J$ we have that $\prod_{j \in L} (e^{-\beta r_{j}} A_{j})^{M_{j}+1}\psi \to 0$ for $M_{j}\to \infty$ for any $j \in L$, so when we consider the limit all terms in the sum except for the one where $L=\emptyset$ vanishes, so:
\begin{equation*}
\sum_{ m \in \mathbb{N}^{J}}e^{-\beta r \cdot (0,m)} A^{(0,m)}\phi = \psi .
\end{equation*}
This implies $\phi\neq 0$ and it implies that we can define a Borel probability measure $\nu$ on $\Lambda^{*}$ by $\nu=\sum_{m \in \mathbb{N}^{J}} \nu^{m}$ that as in the case where $J=\emptyset$ satisfies $\nu(Z(v))=\psi_{v}$ for each $v \in \Lambda^{0}$. Since $\nu$ is a Borel probability measure on the second countable locally compact Hausdorff space $\Lambda^{*}$ it is also a regular measure. We define a state $\omega_{\psi}$ by:
\begin{equation*}
\omega_{\psi}(a)=\int_{\Lambda^{*}} P(a) \ d\nu \qquad \forall a \in \mathcal{T}C^{*}(\Lambda)
\end{equation*}
where $P: \mathcal{T}C^{*}(\Lambda) \to C(\Lambda^{*})$ is the canonical conditional expectation. Since $P(S_{\lambda}S_{\mu}^{*})=0$ when $\mu \neq \lambda$ it follows that $\omega_{\psi}$ is gauge-invariant. For any path $\lambda \in \Lambda^{(n,l)}$ for some $n \in \mathbb{N}^{I}$ and $l \in \mathbb{N}^{J}$ we have by \eqref{ecirkel} when $J\neq \emptyset$:
\begin{align*}
&\nu(Z(\lambda))=\sum_{m \in \mathbb{N}^{J}} \nu^{m}(Z(\lambda))
=\sum_{m \geq l} \nu^{m}(Z(\lambda))
= \sum_{m \geq l} e^{- \beta r \cdot (n,m)} (A^{(0, m-l)} \phi)_{s(\lambda)} \\
&= \sum_{m \in \mathbb{N}^{J}} e^{- \beta r \cdot (n,m+l)} (A^{(0, m)} \phi)_{s(\lambda)}
= e^{- \beta r \cdot d(\lambda)}\sum_{m \in \mathbb{N}^{J}} e^{- \beta r \cdot (0,m)} (A^{(0, m)} \phi)_{s(\lambda)} \\
&= e^{- \beta r \cdot d(\lambda)} \nu(Z(s(\lambda)))
\end{align*}
When $J=\emptyset$ we also have $\nu(Z(\lambda))=e^{- \beta r \cdot d(\lambda)} \nu(Z(s(\lambda)))$, so in both cases this implies that:
\begin{equation*}
\omega_{\psi}(S_{\lambda}S_{\mu}^{*}) = \delta_{\lambda, \mu} \nu(Z(\lambda)) 
=\delta_{\lambda, \mu} e^{- \beta r \cdot d(\lambda)} \nu(Z(s(\lambda)))
=\delta_{\lambda, \mu} e^{- \beta r \cdot d(\lambda)} \psi_{s(\lambda)}  
\end{equation*}
It now follows, for example as in the proof of part $(b)$ of Proposition 3.1 in \cite{aHLRS1}, that $\omega_{\psi}$ is a $\beta$-KMS state for $\alpha^{r}$.
\end{proof}

The proof of Proposition \ref{p43} yields the following corollary. 

\begin{cor} \label{c43}
In the setting of Proposition \ref{p43} assume that there exist sets $I,J$ such that $I \sqcup J=\{1, 2, \dots , k\}$ and $e^{-\beta r_{i}}A_{i}\psi =  \psi$ for $i \in I$ and $\lim_{n \to \infty}(e^{-\beta r_{j}}A_{j})^{n} \psi=0$ for $j \in J$. Then the measure $m_{\psi}$ on $\Lambda^{*}$ associated to $\omega_{\psi}$ is concentrated on $\partial^{I}\Lambda$, i.e. $m_{\psi}(\partial^{I}\Lambda)=1$.
\end{cor}

\section{Decomposition of gauge-invariant KMS states} \label{decompofKMS}
In this section we will investigate the gauge-invariant KMS states by analysing the sub-invariant vectors. The first step in this analysis is to construct sub-invariant vectors using components in different equivalences $\sim_{I}$ in the $k$-graph. The next step is to prove that all invariant vectors can be realised as convex combinations of the invariant vectors constructed.

First let us introduce some notation. For a set $S \subseteq \Lambda^{0}$ and $B\in M_{\Lambda^{0}}(\mathbb{R})$ we let $B^{S}\in M_{S}(\mathbb{R})$ denote the restriction of $B$ to $S \times S$ and for any matrix $B$ we write $\rho(B)$ for its spectral radius. Whenever we have a $k$-graph $\Lambda$ with vertex matrices $A_{1}, \dots , A_{k}$ and some $S\subseteq \Lambda^{0}$ we set:
\begin{equation*}
\rho(A^{S}):= (\rho(A_{1}^{S}), \rho(A_{2}^{S}), \dots , \rho(A_{k}^{S} )) \in \mathbb{R}^{k}
\end{equation*}

\begin{defn} \label{d51}
Let $\Lambda$ be a finite $k$-graph, $r\in \mathbb{R}^{k}$, $\beta \in \mathbb{R}$ and let $I \subseteq \{1, \dots , k\}$. A component $C$ in $\Lambda_{I}$ (i.e. an equivalence class for $\sim_{I}$) is called a $(I,\beta , r )$-subharmonic component, if it satisfies:
\begin{enumerate}
\item \label{subh1} All equivalence classes $D$ in $\sim_{I}$ with $D\neq C$ and $D\subseteq \overline{C}^{I}$ satisfies:
\begin{equation*}
\rho(A^{D})_{I} \lneq \rho(A^{C})_{I}
\end{equation*}
\item \label{subh2} $\rho(A_{i}^{C})=e^{\beta r_{i}}$ for $i\in I$.
\item \label{subh3}$\rho(A_{j}^{\overline{C}}) <e^{\beta r_{j}}$ for $j\in J:=\{1, \dots , k\}\setminus I$.
\end{enumerate}
\end{defn}
When $I= \emptyset$ then $\Lambda_{I} =\Lambda^{0}$ and the different equivalence classes are just the sets $\{v\}$, $v\in \Lambda^{0}$, so condition \ref{subh3} is the only one that is not trivially fulfilled. We will need some results from \cite{C} regarding the construction of vectors over $\Lambda^{0}$ which we will summarise in the following Lemma \ref{l502}.

\begin{lemma} \label{l502}
Let $\Lambda$ be a finite $k$-graph and let $r\in \mathbb{R}^{k}$ and $\beta \in \mathbb{R}$. For each $(\{1, \dots , k\}, \beta ,r)$-subharmonic component $C$ there exists a unique vector $z^{C}\in [0, \infty[^{\Lambda^{0}}$ of unit $1$-norm satisfying \ref{22222}. and \ref{22}.:
\begin{enumerate}
\item \label{22222}$z^{C}_{v} =0$ for $v \notin \overline{C}$.
\item \label{22}$A_{i} z^{C} = e^{\beta r_{i}} z^{C}$ for all $i\in \{1, \dots , k\}$.
\end{enumerate}
Furthermore $z^{C}_{v} >0$ for $v\in \overline{C}$. For any $x\in [0, \infty[ ^{\Lambda^{0}}$ of unit $1$-norm with $A_{i} x=e^{\beta r_{i}} x$ for all $i\in \{1, \dots , k\}$ there is a unique collection of $(\{1, \dots , k\}, \beta ,r)$-subharmonic components $\mathcal{C}$ in $\Lambda$ and numbers $t_{C}>0$, $C\in \mathcal{C}$, such that:
\begin{equation*}
x=\sum_{C\in \mathcal{C}} t_{C}z^{C} .
\end{equation*}
\end{lemma}

\begin{proof}
Since the construction of the vectors in \cite{C} is for graphs with no sources, we will start by proving the Lemma when $\Lambda$ is without sources. Let $C$ be a $(\{1, \dots , k\}, \beta ,r)$-subharmonic component in $\Lambda$, then $C$ satisfies the criterion in Lemma 7.11 in \cite{C}. Choosing a finite set $F\subseteq \mathbb{N}^{k}\setminus \{0\} $ with the property that for all $v,w\in \Lambda^{0}$ then $\sum_{n\in F}A^{n}(v,w)>0$ if and only if $v\Lambda^{l}w \neq \emptyset$ for some $l\in \mathbb{N}^{k} \setminus \{0\}$ (such a set is called \emph{well chosen} in \cite{C}), Corollary 7.10 implies that $C$ in the terminology of \cite{C} is \emph{$F$-harmonic}. By Lemma 7.6 in \cite{C} a $F$-harmonic component gives rise to a unique vector $\chi^{C} \in [0, \infty[^{\Lambda^{0}}$ of unit $1$-norm, and by Lemma 7.6 and Lemma 7.7 $\chi^{C}$ satisfies \ref{22222} and \ref{22} and $\chi_{v}^{C}>0$ for $v\in \overline{C}$, proving existence of $z^{C}$. If $z'\in [0, \infty[^{\Lambda^{0}}$ is a vector of unit $1$-norm satisfying \ref{22222} and \ref{22}, then by Proposition 7.9 in \cite{C} there is a unique collection of $F$-harmonic components $\mathcal{C}$ such that $z'$ is a convex combination of the vectors $\chi^{D}$, $D\in \mathcal{C}$, and furthermore $\rho(A^{D})=e^{\beta r}$ for each $D\in \mathcal{C}$. Combining \ref{22222} and the fact that $\chi^{D}$ is positive on $\overline{D}$, we get that each $D\in \mathcal{C}$ satisfies $D\subseteq \overline{C}$, but then condition \ref{subh1} and \ref{subh2} in Definition \ref{d51} combined with $\rho(A^{D})=e^{\beta r}$ imply that $\mathcal{C}=\{C\}$, so $z'=\chi^{C}$, proving uniqueness. For the unique decomposition of $x$, notice that by Lemma 7.11 in \cite{C} a component $C$ is $(\{1, \dots , k\}, \beta ,r)$-subharmonic if and only if $\rho(A^{C})=e^{\beta r}$ and $C$ is $F$-harmonic. The statement therefore follows from Proposition 7.9 in \cite{C} and the construction of the vectors $z^{C}$.

Assume now that $\Lambda$ is a general finite $k$-graph, and let $C$ be a $(\{1, \dots , k\}, \beta ,r)$-subharmonic component in $\Lambda$. Then $A_{i}^{C} \neq 0$ for all $i$, so taking a $v\in C$ and a $i\in \{1, \dots , k\}$ there is a $\mu \in v\Lambda C$ with $d(\mu)_{i} > 0$. The factorisation property then implies that $v\Lambda^{e_{i}}C \neq \emptyset$. Since this is true for all $v\in C$, it follows that $C \subseteq \widetilde{\Lambda^{0}}:=\Lambda^{0}(\{1, \dots , k\})$, and hence by Lemma \ref{l1} it follows that $\overline{C} \subseteq \widetilde{\Lambda^{0}}$. Since $\widetilde{\Lambda^{0}}$ is closed we can consider the finite $k$-graph $\widetilde{\Lambda}:=\Lambda \widetilde{\Lambda^{0}}$, which has vertex matrices $A_{1}^{\widetilde{\Lambda^{0}}}, \dots , A_{k}^{\widetilde{\Lambda^{0}}}$. To see that $\widetilde{\Lambda}$ has no sources take $v\in \widetilde{\Lambda^{0}}$, $m\in \mathbb{N}^{k}$ and $\lambda_{l} \in v\Lambda^{le_{1}+\cdot +le_{k}+m}$ for each $l\in \mathbb{N}$. Since $\lambda_{l}(0,m)\in v\Lambda^{m}$ for each $l\in \mathbb{N}$, there is a $\lambda\in v\Lambda^{m}$ with $\lambda_{l}(0,m)=\lambda$ for infinitely many $l$, which implies that $s(\lambda)\in \widetilde{\Lambda^{0}}$ and hence $v\widetilde{\Lambda}^{m}\neq \emptyset$. Since components in $\widetilde{\Lambda}$ are exactly components in $\Lambda$ contained in $\widetilde{\Lambda^{0}}$, $C$ is a $(\{1, \dots , k\}, \beta ,r)$-subharmonic component in $\tilde{\Lambda}$, so there exists a unique vector $\tilde{z}^{C}\in [0, \infty[^{\widetilde{\Lambda^{0}}}$ of unit $1$-norm with $\tilde{z}^{C}_{v}=0$ when $v\in \widetilde{\Lambda^{0}} \setminus\overline{C}$ and $A_{i}^{\widetilde{\Lambda^{0}}}\tilde{z}^{C}=e^{\beta r_{i}}\tilde{z}^{C}$ for all $i$. Furthermore $\tilde{z}^{C}_{v} >0$ for $v\in \overline{C}$. It is now straightforward to check that defining $z^{C} \in [0, \infty[^{\Lambda^{0}}$ by $z^{C} |_{\widetilde{\Lambda^{0}}}=\tilde{z}^{C}$ and $z^{C}_{v}=0$ for $v\notin \widetilde{\Lambda^{0}}$ gives the desired vector.

Assume $z'\in [0, \infty[^{\Lambda^{0}}$ satisfies \ref{22222} and \ref{22} and is of unit $1$-norm, then  $z'|_{\widetilde{\Lambda^{0}}}\in [0, \infty[^{\widetilde{\Lambda^{0}}}$ also has unit $1$-norm. By \ref{22222} $(z'|_{\widetilde{\Lambda^{0}}})_{v}=0$ for $v\in \widetilde{\Lambda^{0}} \setminus\overline{C}$ and: 
\begin{equation*}
A_{i}^{\widetilde{\Lambda^{0}}}z'|_{\widetilde{\Lambda^{0}}}= (A_{i} z')|_{\widetilde{\Lambda^{0}}} =e^{\beta r_{i}}z'|_{\widetilde{\Lambda^{0}}} \qquad \text{ for all } i.
\end{equation*}
It follows that $z'|_{\widetilde{\Lambda^{0}}}=\tilde{z}^{C}$, which proves uniqueness.

For the last statement let $x\in [0, \infty[^{\Lambda^{0}}$ of unit $1$-norm satisfy $A_{i}x=e^{\beta r_{i}} x$ for all $i$. If $v\Lambda^{n} = \emptyset$ for some $n\in \mathbb{N}$ then $x_{v}=e^{-\beta r\cdot n} (A^{n}x)_{v}=0$, so $x_{v}=0$ for $v\notin\widetilde{\Lambda^{0}}$ and hence $A_{i}^{\widetilde{\Lambda^{0}}} x|_{\widetilde{\Lambda^{0}}}=e^{\beta r_{i}} x|_{\widetilde{\Lambda^{0}}}$ for all $i$. Using the Lemma on $x|_{\widetilde{\Lambda^{0}}}$ we get a unique collection $\mathcal{C}$ of $(\{1, \dots , k\}, \beta, r)$-subharmonic component in $\tilde{\Lambda}$ with corresponding unique vectors $\tilde{z}^{C}$, $C\in \mathcal{C}$, and numbers $t_{C}>0$, $C\in \mathcal{C}$, such that
\begin{equation*}
x|_{\widetilde{\Lambda^{0}}} = \sum_{C \in \mathcal{C}} t_{C} \tilde{z}^{C} .
\end{equation*}
Since $C$ is an $(\{1, \dots , k\}, \beta, r)$-subharmonic component in $\tilde{\Lambda}$ if and only if it is a $(\{1, \dots , k\}, \beta, r)$-subharmonic component in $\Lambda$, it follows from the definition of $z^{C}$ that we have a unique decomposition: 
\begin{equation*}
x = \sum_{C \in \mathcal{C}} t_{C} z^{C}
\end{equation*}
which proves the Lemma.
\end{proof}

\begin{lemma}\label{l52}
Let $\Lambda$ be a finite $k$-graph, $r\in \mathbb{R}^{k}$, $\beta \in \mathbb{R}$, $I \subseteq \{1, \dots , k\}$ and $C$ be a $(I,\beta , r )$-subharmonic component. There exists a unique vector $x^{C}\in [0, \infty[^{\Lambda^{0}}$ of unit $1$-norm satisfying \ref{ee1}. and \ref{ee2}.:
\begin{enumerate}
\item \label{ee1}$x^{C}_{v} =0$ for $v \notin \overline{C}^{I}$.
\item \label{ee2}$A_{i} x^{C} = e^{\beta r_{i}} x^{C}$ for all $i\in I$.
\end{enumerate}
Furthermore $x_{v}^{C}>0$ for $v\in \overline{C}^{I}$. 
\end{lemma}
\begin{proof}
If $I=\emptyset$ then $C=\{v\}$ for some $v\in \Lambda^{0}$, and $x^{C}$ is the vector with $x^{C}_{w}=0$ for $w\neq v$ and $x^{C}_{v}=1$. If $I\neq\emptyset$ consider the finite graph $\Lambda_{I}$ with vertex matrices $(A_{i})_{i\in I}$. Setting $r_{I}=(r_{i})_{i\in I}\in \mathbb{R}^{I}$, it follows from Definition \ref{d51} that $C$ is a $(\{i\}_{i\in I}, \beta , r_{I})$-subharmonic component in the $I$-graph $\Lambda_{I}$, and hence we get the unique vector from Lemma \ref{l502}.
\end{proof}

\begin{prop}\label{p53}
Let $\Lambda$ be a finite $k$-graph, $r\in \mathbb{R}^{k}$, $\beta \in \mathbb{R}$, $I \sqcup J = \{1, \dots , k\}$ be a partition and $C$ be a $(I,\beta , r )$-subharmonic component. Denote by $x^{C} \in [0, \infty[^{\Lambda^{0}}$ the unique vector given in Lemma \ref{l52} using $C$. Set:
\begin{equation}\label{etildex}
\tilde{x}^{C}\vert_{\overline{C}}:=\prod_{j\in J}(1_{\overline{C}}-e^{-\beta r_{j}} A_{j}^{\overline{C}})^{-1} x^{C}\vert_{\overline{C}}
\end{equation}
and $\tilde{x}^{C}\vert_{\Lambda^{0}\setminus \overline{C}}=0$. Then $\tilde{x}^{C}$ is sub-invariant for the family $\{e^{-\beta r_{i}}A_{i}\}_{i=1}^{k}$.
\end{prop}

\begin{proof}
If $J=\emptyset$ then $\tilde{x}^{C}=x^{C}$ which is clearly invariant for $\{e^{-\beta r_{i}}A_{i}\}_{i=1}^{k}$, so assume $J\neq \emptyset$. Notice first that condition \ref{subh3} in Definition \ref{d51} implies that $(1_{\overline{C}}-e^{-\beta r_{j}} A_{j}^{\overline{C}})^{-1} $ exists for each $j\in J$, so \eqref{etildex} makes sense. To express $\tilde{x}^{C}$ differently, assume that $J_{0}\subseteq J$ is an arbitrary non-empty subset, then for any $N \in \mathbb{N}^{J_{0}}$ we have that:
\begin{equation*}
\sum_{0\leq n \leq N} \prod_{j \in J_{0}}e^{-\beta r_{j}n_{j}}(A_{j}^{\overline{C}})^{n_{j}}
= \prod_{j \in J_{0}}\left(  \sum_{ n_{j}=0}^{N_{j}}  e^{-\beta r_{j}n_{j}}(A_{j}^{\overline{C}})^{n_{j}}\right)
\end{equation*} 
Hence as $N \to \infty$ in $\mathbb{N}^{J_{0}}$ we get that:
\begin{equation} \label{esum}
\prod_{j\in J_{0}}(1_{\overline{C}}-e^{-\beta r_{j}} A_{j}^{\overline{C}})^{-1}
=\sum_{n \in \mathbb{N}^{J_{0}}} \prod_{j \in J_{0}}e^{-\beta r_{j}n_{j}}(A_{j}^{\overline{C}})^{n_{j}}
\end{equation} 
Let $L \subseteq \{1, \dots , k\}$, to prove that $\tilde{x}^{C}$ is sub-invariant we then want to verify \eqref{eaiv} for the set $L$ and the family $\{e^{-\beta r_{i}}A_{i}\}_{i=1}^{k}$. Since $x^{C}\in [0, \infty[^{\Lambda^{0}}$ it follows from \eqref{esum} with $J=J_{0}$ that $\tilde{x}^{C}\in [0, \infty[^{\Lambda^{0}}$, proving \eqref{eaiv} when $L=\emptyset$. Assume then that $L\neq \emptyset$. If $y\in [0, \infty[^{\Lambda^{0}}$ is a vector with $y\lvert_{\Lambda^{0}\setminus\overline{C}}=0$ and $B\in M_{\Lambda^{0}}(\mathbb{R}_{+})$ has the property that $B(v,w)>0$ implies $v\leq w$, then it follows that:
\begin{equation} \label{eBy}
(By)\lvert_{\Lambda^{0}\setminus\overline{C}}=0
\quad , \quad (By)\lvert_{\overline{C}} = B^{\overline{C}} (y\lvert_{\overline{C}} )
\end{equation}
This implies that $(A^{n} \tilde{x}^{C})\vert_{\Lambda^{0}\setminus \overline{C}}=0$ for all $n\in \mathbb{N}^{k}$, and hence for $v\in \Lambda^{0}\setminus \overline{C}$ we get
\begin{equation*}
\left[\prod_{l \in L} (1_{\Lambda^{0}} -e^{-\beta r_{l}}A_{l})\tilde{x}^{C}\right]_{v} 
= \left[\sum_{S \subseteq L}(-1)^{\lvert S \rvert}\prod_{l \in S} e^{-\beta r_{l}}A_{l}\tilde{x}^{C}\right]_{v}
= \tilde{x}^{C}_{v} \geq 0.
\end{equation*}
Using the second equality in \eqref{eBy} we obtain:
\begin{align} \label{a1}
\nonumber\left[\prod_{l \in L} (1_{\Lambda^{0}} -e^{-\beta r_{l}}A_{l})\tilde{x}^{C}\right]_{\overline{C}} 
&= \left[\sum_{S \subseteq L}(-1)^{\lvert S \rvert}\prod_{l \in S} e^{-\beta r_{l}}A_{l}\tilde{x}^{C}\right]_{\overline{C}} \\
\nonumber&= \sum_{S \subseteq L}(-1)^{\lvert S \rvert}\prod_{l \in S} e^{-\beta r_{l}}A_{l}^{\overline{C}}(\tilde{x}^{C}\vert_{\overline{C}})\\
&= \prod_{l \in L} (1_{\overline{C}} -e^{-\beta r_{l}}A_{l}^{\overline{C}})(\tilde{x}^{C}\vert_{\overline{C}})
\end{align}
It now follows from \eqref{etildex} that if there is a $i\in L\cap I$, then since $A_{i}^{\overline{C}}x^{C}\vert_{\overline{C}}=(A_{i}x^{C})\vert_{\overline{C}}=e^{\beta r_{i}} x^{C}\vert_{\overline{C}}$ we get that $(1_{\overline{C}} -e^{-\beta r_{i}}A_{i}^{\overline{C}}) \tilde{x}^{C}\vert_{\overline{C}}=0$, and hence the expression in \eqref{a1} is zero. If $L\cap I = \emptyset$ then $L\subseteq J$, and:
\begin{equation*}
\prod_{l \in L} (1_{\overline{C}} -e^{-\beta r_{l}}A_{l}^{\overline{C}})(\tilde{x}^{C}\vert_{\overline{C}})
=\prod_{j\in J\setminus L}(1_{\overline{C}}-e^{-\beta r_{j}} A_{j}^{\overline{C}})^{-1} x^{C}\vert_{\overline{C}}
\end{equation*}
It follows from \eqref{esum} with $J_{0}=J\setminus L$ that this is a non-negative vector, and combined with \eqref{a1} this implies that $\tilde{x}^{C}$ is sub-invariant.
\end{proof}

\begin{defn} \label{d54}
When $C$ is a $(I,\beta , r )$-subharmonic component we set $y^{C}:=\tilde{x}^{C} / \lVert \tilde{x}^{C} \rVert_{1}$.
\end{defn}
The notation in Definition \ref{d54} is not well defined since a set $C \subseteq \Lambda^{0}$ can both be a $(I,\beta , r )$-subharmonic component and a $(I',\beta' , r')$-subharmonic component with $(I,\beta , r )\neq (I',\beta' , r')$. If however $C$ is $(I,\beta , r )$-subharmonic and $i\in I$ then $A_{i}(v,w)=0$ for $v\in \overline{C}$ and $w\in \overline{C}^{I}$, so we get that $\rho(A_{i}^{\overline{C}}) \geq \rho(A_{i}^{\overline{C}^{I}})$, and since $A_{i}(v,w)=0$ for $v\in C$ and $w\in \overline{C}^{I}$ we furthermore get that $\rho(A_{i}^{\overline{C}^{I}}) \geq \rho(A_{i}^{C})$, so by Definition \ref{d51} $C$ can not be $(I',\beta , r )$-subharmonic for an $I'\neq I$. Since we will formulate our results for some fixed values of $r$ and $\beta$, we therefore abuse notation and simply write $y^{C}$.

Proposition \ref{p53} implies that a $(I,\beta ,r)$-subharmonic component gives rise to a gauge-invariant $\beta$-KMS state $\omega$ for $\alpha^{r}$. To prove that all gauge-invariant states are given by convex combinations of states arising from such components, it becomes essential that we can discover the vector $x^{C}$ from $\omega$. To do this we need the following technical result.

\begin{lemma} \label{l55}
Let $\Lambda$ be a finite $k$-graph and let $\omega$ be a $\beta$-KMS state for $\alpha^{r}$ for some $r\in \mathbb{R}^{k}$ and $\beta \in \mathbb{R}$. Let $m$ be the measure on $\Lambda^{*}$ associated to $\omega$ and $I\sqcup J\subseteq \{1, \dots , k\}$ be some partition. For each $\lambda \in \Lambda $ the set:
\begin{equation*}
\lambda\Lambda^{\infty_{I},0} =\{x \in \Lambda^{*} \ : \ x=\lambda x' \text{ for some } x'\in  \Lambda^{\infty_{I},0} \} 
\end{equation*}  
is Borel and
\begin{equation*}
m(\lambda\Lambda^{\infty_{I},0})=e^{-\beta r\cdot d(\lambda)}m(s(\lambda)\Lambda^{\infty_{I},0}) .
\end{equation*}
\end{lemma}

\begin{proof}
To see that $\lambda\Lambda^{\infty_{I},0}$ is Borel set $p:=1^{I}\in \mathbb{N}^{I}$ if $I\neq \emptyset$ and set $p=0$ if $I=\emptyset$, then:
\begin{equation*}
\lambda\Lambda^{\infty_{I},0} =\bigcap_{n \in \mathbb{N}} \bigcup_{\mu\in s(\lambda)\Lambda^{n\cdot p}} \left[
Z(\lambda \mu) \setminus\left(  \bigcup_{j \in J} \bigcup_{e\in s(\lambda)\Lambda^{e_{j}} }Z(\lambda e) \right) \right]
\end{equation*}
Since we take the union over decreasing sets we get that:
\begin{align} \label{ep511}
\nonumber m(\lambda\Lambda^{\infty_{I},0})&= \lim_{n \to \infty} \sum_{\mu\in s(\lambda)\Lambda^{n\cdot p}}m\left(Z(\lambda \mu) \setminus\left(  \bigcup_{j \in J} \bigcup_{e\in s(\lambda)\Lambda^{e_{j}} }Z(\lambda e) \right) \right) \\
&=\lim_{n \to \infty} \sum_{\mu\in s(\lambda)\Lambda^{n\cdot p}}m\left(Z(\lambda \mu) \setminus\left(  \bigcup_{j \in J} \bigcup_{e\in s(\mu)\Lambda^{e_{j}} }Z(\lambda \mu e) \right) \right) 
\end{align}
Set $e_{L}=\sum_{l\in L} e_{l}$ for any $L \subseteq J$ and $\psi_{v}=\omega(p_{v})$ for $v\in \Lambda^{0}$. We claim that for any path $\eta \in \Lambda$:
\begin{equation} \label{ep512}
m\left(Z(\eta) \setminus\left(  \bigcup_{j \in J} \bigcup_{e\in s(\eta)\Lambda^{e_{j}} }Z(\eta e) \right)\right)
 = \sum_{L \subseteq J} (-1)^{\lvert L \rvert} e^{-\beta r \cdot d(\eta)} e^{-\beta r\cdot e_{L}}
 (A^{e_{L}}\psi)_{s(\eta)}
\end{equation}
The Lemma follows from \eqref{ep512}, because using it twice on \eqref{ep511} yields:
\begin{align*}
m(\lambda\Lambda^{\infty_{I},0})&= \lim_{n \to \infty} \sum_{\mu\in s(\lambda)\Lambda^{n\cdot p}}\sum_{L \subseteq J} (-1)^{\lvert L \rvert} e^{-\beta r \cdot d(\lambda \mu)} e^{-\beta r\cdot e_{L}}
 (A^{e_{L}}\psi)_{s(\mu)} \\
&= e^{-\beta r \cdot d(\lambda)}\lim_{n \to \infty} \sum_{\mu\in s(\lambda)\Lambda^{n\cdot p}}\sum_{L \subseteq J} (-1)^{\lvert L \rvert} e^{-\beta r \cdot d( \mu)} e^{-\beta r\cdot e_{L}}
 (A^{e_{L}}\psi)_{s(\mu)}  \\
&= e^{-\beta r \cdot d(\lambda)} m(s(\lambda)\Lambda^{\infty_{I},0})
\end{align*}
To prove \eqref{ep512} set $\mathcal{M}(e_{j}):=\bigcup_{e\in s(\eta)\Lambda^{e_{j}} }Z(\eta e)$. We use that $Z(\eta e)\subseteq Z(\eta)$ for each $e\in s(\eta)\Lambda^{e_{j}}$ and $j\in J$ to get the equality:
\begin{equation} \label{ep513}
1_{Z(\eta) \setminus\left(  \bigcup_{j \in J} \mathcal{M}(e_{j}) \right)}
= \prod_{j \in J}(1_{Z(\eta)}-1_{\mathcal{M}(e_{j})})
=\sum_{L\subseteq J} (-1)^{\lvert L \rvert} \prod_{j \in L} 1_{\mathcal{M}(e_{j})}
\end{equation}
Since $\prod_{j \in L} 1_{\mathcal{M}(e_{j})}=1_{\bigcap_{j\in L}\mathcal{M}(e_{j})}$ we get \eqref{ep512} by combining \eqref{ep513} with:
\begin{align*}
m\left(\bigcap_{j\in L}\mathcal{M}(e_{j})\right)
&=m\left(\bigcup_{e\in s(\eta)\Lambda^{e_{L}} }Z(\eta e)\right)
= \sum_{e\in s(\eta)\Lambda^{e_{L}} } m(Z(\eta e))\\
&=e^{-\beta r \cdot (d(\eta)+e_{L})} (A^{e_{L}}\psi)_{s(\eta)}
\end{align*}
\end{proof}

\begin{lemma} \label{l56}
Let $\Lambda$ be a finite $k$-graph, $r\in \mathbb{R}^{k}$ and $\beta \in \mathbb{R}$. Let $I \subseteq \{1, \dots , k\}$ and $C$ be a $(I,\beta , r )$-subharmonic component. Let $\omega$ be the KMS state associated to the vector $y^{C}$, and let $m_{C}$ be the measure associated to $\omega$, then
\begin{equation*}
 x^{C}_{v} =\lVert\tilde{x}^{C} \rVert_{1} m_{C}(v\Lambda^{\infty_{I}, 0}) \qquad \text{ for all } v \in \Lambda^{0}.
\end{equation*}
\end{lemma}

\begin{proof}
Since $m_{C}(\partial^{I} \Lambda)=1$ by Corollary \ref{c43}, the formula \eqref{ep512} implies:
\begin{equation*}
m_{C}(v\Lambda^{\infty_{I}, 0})
=m_{C} \left( Z(v)\setminus \left( \bigcup_{j\in J} \bigcup_{e\in v\Lambda^{e_{j}}} Z(e) \right)    \right)
=\sum_{L \subseteq J} (-1)^{\lvert L \rvert} e^{-\beta r\cdot e_{L}} (A^{e_{L}} y^{C})_{v}
\end{equation*}
for each $v\in \Lambda^{0}$. When $v\notin \overline{C}$ then $(A^{e_{L}} y^{C})_{v}=0$ for each $L$, and hence $m_{C}(v\Lambda^{\infty_{I}, 0})=0$. When $v\in \overline{C}$ then $(A^{e_{L}} y^{C})_{v}=((A^{\overline{C}})^{e_{L}} y^{C}|_{\overline{C}})_{v}$ and hence:
\begin{align*}
m_{C}(v\Lambda^{\infty_{I}, 0}) 
&=\sum_{L \subseteq J} (-1)^{\lvert L \rvert}
e^{-\beta r\cdot e_{L}} ((A^{\overline{C}})^{e_{L}} y^{C}|_{\overline{C}})_{v} \\
&=\lVert \tilde{x}^{C} \rVert_{1}^{-1}\left[\prod_{j \in J} (1_{\overline{C}}-e^{-\beta r_{j}}A_{j}^{\overline{C}}) \tilde{x}^{C}\vert_{\overline{C}}\right]_{v} 
=\lVert \tilde{x}^{C} \rVert_{1}^{-1} x^{C}_{v}
\end{align*}
Since $x^{C}_{v}=0$ for $v\notin \overline{C}$ this proves the Lemma.
\end{proof}

By Proposition \ref{p32} we already have a decomposition of a general sub-invariant vector, so we can focus on decomposing the vectors appearing in Proposition \ref{p32}.

\begin{prop} \label{p57}
Let $\Lambda$ be a finite $k$-graph, $r\in \mathbb{R}^{k}$ and $\beta \in \mathbb{R}$. Let $\psi \in [0, \infty[^{\Lambda^{0}}$ be a sub-invariant vector for the family $\{e^{-\beta r_{i}} A_{i}\}_{i=1}^{k}$ and assume that there is a partition $I \sqcup J= \{1, \dots , k\}$ such that $e^{-\beta r_{i}}A_{i}\psi=\psi$ for $i \in I$ and $\lim_{n \to \infty}(e^{-\beta r_{j}}A_{j})^{n}\psi=0$ for $j \in J$. There exists a unique collection of $(I, \beta, r)$-subharmonic components $\mathcal{C}$ and numbers $t_{C}>0$, $C \in \mathcal{C}$, such that
\begin{equation*}
\psi = \sum_{C \in \mathcal{C}} t_{C} y^{C} .
\end{equation*}
\end{prop}

\begin{proof}
Let $\omega_{\psi}$ be the gauge-invariant $\beta$-KMS state for $\alpha^{r}$ given by $\psi$, and let $m_{\psi}$ be the associated measure on $\Lambda^{*}$. Define a vector $\psi'\in [0, \infty[^{\Lambda^{0}}$ by:
\begin{equation*}
\psi'_{v} = m_{\psi}(v\Lambda^{\infty_{I}, 0}) \qquad \text{for } v\in \Lambda^{0}.
\end{equation*}
If $I=\emptyset$ then we can uniquely write $\psi'$ as in \eqref{edeo} below where each $z^{C}$ is the indicator function for the $v$ with $C=\{v\}$. When $I\neq \emptyset$ it follows from Lemma \ref{l55} that for each $i\in I$ and $v\in \Lambda^{0}$:
\begin{equation*}
\psi'_{v} 
= m_{\psi}\left(\bigcup_{\mu \in v\Lambda^{e_{i}}} \mu\Lambda^{\infty_{I}, 0}\right)
=e^{-\beta r_{i}} \sum_{w \in \Lambda^{0}} A_{i}(v,w) m_{\psi}(w\Lambda^{\infty_{I}, 0})
=e^{-\beta r_{i}} (A_{i}\psi')_{v}
\end{equation*}
So $A_{i}\psi'=e^{\beta r_{i}}\psi'$ for each $i\in I$, and if $\psi'\neq 0$ considering the graph $\Lambda_{I}$ and the action given by $r_{I}=(r_{i})_{i\in I}$, Lemma \ref{l502} gives us a unique collection $\mathcal{C}$ of $(\{ i\}_{i\in I}, \beta, r_{I})$-subharmonic components in $\Lambda_{I}$ and numbers $t_{C}'>0$ for $C\in \mathcal{C}$ such that:
\begin{equation} \label{edeo}
\psi'= \sum_{C \in \mathcal{C}} t'_{C} z^{C}.
\end{equation}
If $\psi'=0$, we set $\mathcal{C}=\emptyset$ and \eqref{edeo} holds true.

We now want to prove that $e^{\beta r_{j}}> \rho(A_{j}^{\overline{C}})$ for all $j \in J$ and $C \in \mathcal{C}$, since that would imply that each $C \in \mathcal{C}$ was $(I, \beta , r)$-subharmonic and that each vector $z^{C}$ from \eqref{edeo} was equal to $x^{C}$ from Lemma \ref{l52}. When $J=\emptyset$ this is trivial, so assume this is not the case. Since $m_{\psi}(\partial^{I}\Lambda)=1$, we get for any $v \in \Lambda^{0}$ that:
\begin{equation*}
\psi_{v}=m_{\psi}(Z(v))= m_{\psi}\left(\bigcup_{n \in \mathbb{N}^{J}}v\Lambda^{\infty_{I}, n}\right)
=\sum_{n \in \mathbb{N}^{J}}m_{\psi}\left( v\Lambda^{\infty_{I}, n}\right)
\end{equation*}
Looking at just one of the terms in the sum and using Lemma \ref{l55} we get:
\begin{align*}
&m_{\psi}\left(v\Lambda^{\infty_{I}, n}\right) 
=\sum_{w \in \Lambda^{0}}m_{\psi}\left(\bigcup_{\mu\in v  \Lambda^{(0,n)}w} \mu \Lambda^{\infty_{I}, 0}  \right)
=\sum_{w \in \Lambda^{0}} \sum_{\mu \in v  \Lambda^{(0,n)}w} m_{\psi}\left( \mu \Lambda^{\infty_{I}, 0}  \right)  \\
&=\sum_{w \in \Lambda^{0}}  e^{-\beta r \cdot (0,n)} A^{(0,n)}(v,w)\psi'_{w}
= e^{-\beta r \cdot (0,n)} \left(A^{(0,n)}\psi' \right)_{v}
\end{align*}
Let $v \in \overline{C}$ and $w\in \overline{C}^{I}$ for a $C\in \mathcal{C}$, then $\psi'_{w}> 0$, and by the above calculations:
\begin{equation*}
\psi_{v}=\sum_{n \in \mathbb{N}^{J}}\sum_{u \in \Lambda^{0}}  e^{-\beta r \cdot (0,n)} A^{(0,n)}(v,u)\psi'_{u}
\end{equation*}
This implies that $\sum_{n \in \mathbb{N}^{J}} e^{-\beta r\cdot (0,n) } A^{(0,n)}(v,w)<\infty$ for such $w$ and $v$. Now let $u\in \overline{C}$, then there exists a $m\in \mathbb{N}^{J}$ and $w \in \overline{C}^{I}$ such that $A^{(0,m)}(u,w)>0$. For each $n \in \mathbb{N}^{J}$ and $v \in \overline{C}$ we have:
\begin{equation*}
e^{-\beta r\cdot (0,n)}A^{(0,n)}(v,u) A^{(0,m)}(u,w) \leq e^{\beta r \cdot (0,m)}\left(e^{-\beta r\cdot (0,n+m)}A^{(0,n+m)}(v,w) \right)
\end{equation*}
So $\sum_{n \in \mathbb{N}^{J}} e^{-\beta r\cdot (0,n) } A^{(0,n)}(v,u)<\infty$ for all $v,u\in \overline{C}$, and hence the sum $\sum_{l=0}^{\infty} (e^{-\beta r_{j}} A_{j}^{\overline{C}})^{l}$ converges for each $j\in J$, proving $\rho( A_{j}^{\overline{C}})<e^{\beta r_{j} }$.

We now know that each $C\in \mathcal{C}$ is a $(I, \beta ,r)$-subharmonic component, so the vectors $y^{C}$ exist for each such $C$. For any $v\in \Lambda^{0}$:
\begin{align*}
\psi_{v}&= \sum_{n \in \mathbb{N}^{J}}e^{-\beta r \cdot (0,n)} \left(A^{(0,n)}\psi' \right)_{v}
=\sum_{C \in \mathcal{C}} t_{C}' \sum_{n \in \mathbb{N}^{J}}e^{-\beta r \cdot (0,n)} \left(A^{(0,n)}x^{C} \right)_{v} \\
&= \sum_{C \in \mathcal{C}, v\in \overline{C}} t'_{C} \left[\left(\sum_{n \in \mathbb{N}^{J}} \prod_{j\in J} \left(e^{-\beta r_{j} }A_{j}^{\overline{C}}\right)^{n_{j}} \right)x^{C}\vert_{\overline{C}} \right]_{v}
=\sum_{C \in \mathcal{C}} t'_{C}\tilde{x}^{C}_{v}
=\sum_{C \in \mathcal{C}} t_{C}y^{C}_{v}
\end{align*}
where $t_{C}=t'_{C} \lVert \tilde{x}^{C} \rVert_{1}>0$. We have now proved that the decomposition exists.

\bigskip

To prove the uniqueness statement assume that $\mathcal{D}$ is a collection of $(I, \beta, r)$-subharmonic components and that there exists $s_{D}>0$ for each $D \in \mathcal{D}$ such that:
\begin{equation*}
\psi=\sum_{D\in \mathcal{D}} s_{D} y^{D} .
\end{equation*}
Let $m_{D}$ be the measure on $\Lambda^{*}$ associated to $y^{D}$ for each $D\in \mathcal{D}$, since $m_{\psi}=\sum_{D\in \mathcal{D}} s_{D} m_{D}$ it follows by Lemma \ref{l56} that considering these measures on $\Lambda^{\infty_{I}, 0}$ give:
\begin{equation*}
\psi'= \sum_{D\in \mathcal{D}} s_{D} \lVert \tilde{x}^{D} \rVert_{1}^{-1} x^{D}.
\end{equation*}
When $I\neq \emptyset$ then $\mathcal{D}$ can be considered a collection of $(\{i\}_{i\in I}, \beta, r_{I})$-subharmonic components in $\Lambda_{I}$ and $x^{D}$ are then by construction the unique vectors from Lemma \ref{l502}. For all $I$ uniqueness of the decomposition in \eqref{edeo} of $\psi'$ then gives $\mathcal{D} =\mathcal{C}$ and $s_{C} \lVert \tilde{x}^{C} \rVert_{1}^{-1} = t'_{C}$, and hence $t_{C}=s_{C}$, for each $C \in \mathcal{C}$.
\end{proof}

Combining Proposition \ref{p57}, Proposition \ref{p43} and Proposition \ref{p32} we get the following

\begin{thm} \label{t58}
Let $\Lambda$ be a finite $k$-graph, $r\in \mathbb{R}^{k}$ and $\beta \in \mathbb{R}$. For $I\subseteq \{1, \dots , k\}$ let $\mathcal{C}_{r}^{I}(\beta)$ be the $(I, \beta, r)$-subharmonic components and set:
\begin{equation*}
\mathcal{C}_{r}(\beta):=\bigsqcup_{I\subseteq \{1, \dots ,k\}} \mathcal{C}_{r}^{I}(\beta).
\end{equation*}
There is an affine bijective correspondence between functions $f: \mathcal{C}_{r}(\beta)\to [0,1]$ with $\sum_{C\in \mathcal{C}_{r}(\beta)} f(C)=1$ and the gauge-invariant $\beta$-KMS states for $\alpha^{r}$ on $\mathcal{T}C^{*}(\Lambda)$. A KMS state $\omega$ corresponding to a function $f$ is given by:
\begin{equation*}
\omega(S_{\lambda} S_{\mu}^{*}) = \delta_{\lambda, \mu} e^{-\beta r\cdot d(\lambda)} \psi_{s(\lambda)}
\end{equation*}
where:
\begin{equation*}
\psi = \sum_{C\in \mathcal{C}_{r}(\beta)} f(C) y^{C} .
\end{equation*}
\end{thm}

\begin{remark}
Notice that the face of the simplex of gauge-invariant KMS states given by components in $\mathcal{C}_{r}^{I}(\beta)$ corresponds to the face in the simplex of sup-invariant vectors of unit $1$-norm satisfying $A_{i}x=e^{\beta r_{i}}x$ for $i\in I$ and $(e^{-\beta r_{j}}A_{j})^{l}x\to 0$ for $l\to \infty$ for $j\notin I$, which again corresponds to the face in the simplex of quasi-invariant Borel probability measures $m$ with Radon Nikodym derivative $e^{-\beta c_{r}}$ satisfying $m(\partial^{I} \Lambda)=1$.
\end{remark}

\begin{remark}
Theorem \ref{t58} is already an improvement of the results obtained in \cite{FaHR}. To see this, notice that if $r$ and $\beta$ satisfies condition $2$ mentioned in the introduction and $\omega$ is a $\beta$-KMS state for $\alpha^{r}$, then for any $\lambda, \mu \in \Lambda$ and $t\in \mathbb{R}$ we get:
\begin{equation*}
\omega(S_{\lambda}S_{\mu}^{*})= \omega(\alpha_{t}^{r}(S_{\lambda}S_{\mu}^{*}))
=e^{i t r\cdot(d(\lambda)-d(\mu))}\omega(S_{\lambda}S_{\mu}^{*})
\end{equation*}
If $\omega(S_{\lambda}S_{\mu}^{*})\neq 0$ then this is only possible if $r\cdot(d(\lambda)-d(\mu)) =0$, implying that $d(\lambda)=d(\mu)$. So $S_{\lambda}S_{\lambda}^{*}$ and $S_{\mu}S_{\mu}^{*}$ are mutually orthogonal when $\lambda \neq \mu$, and since $\omega(S_{\lambda}S_{\mu}^{*})= e^{-\beta r\cdot d(\lambda)}\omega(S_{\mu}^{*}S_{\lambda})$, we get that $\omega(S_{\lambda}S_{\mu}^{*})=0$ when $\lambda \neq \mu$. This implies that $\omega$ is gauge-invariant, so Theorem \ref{t58} gives a complete description of the $\beta$-KMS states for finite $k$-graphs satisfying condition $2$ from the introduction.
\end{remark}

\section{Including the non gauge-invariant KMS states} \label{nongauge}

We are now interested in determining the KMS states that are not gauge-invariant. To do this we will use the ideas developed in \cite{C}. Theorem \ref{t58} gives us a complete description of the gauge-invariant KMS states, but by Lemma 3.1 in \cite{C} this is exactly the KMS states $\omega$ satisfying $\omega \circ P = \omega$. So in the terminology of \cite{N} we can consider Theorem \ref{t58} as a description of the quasi-invariant Borel probability measures with Radon-Nikodym cocycle $e^{-\beta c_{r}}$, where $c_{r}$ is the $1$-cocycle $c_{r}(x,n,y):=r \cdot n$. Hence we can use Theorem 5.2 in \cite{C} to obtain a description of all KMS states. We follow the outline and ideas in \cite{C} to do this, and start by analysing the relationship between the paths in $\Lambda^{*}$ and the measures associated to extremal KMS states.

\begin{defn}
We say that a path $x \in \Lambda^{*}$ eventually lies in $S$ for some set $S \subseteq \Lambda^{0}$, if there exists a $n \in \mathbb{N}^{k}$ with $n\leq d(x)$ such that $r(\sigma^{m}(x))\in S$ for all $m \in \mathbb{N}^{k}$ with $n \leq m \leq d(x)$.
\end{defn}

\begin{lemma}
Let $\Lambda$ be a finite $k$-graph, $r\in \mathbb{R}^{k}$, $\beta \in \mathbb{R}$ and let $I \subseteq \{1,2, \cdots , k\}$. If $D$ is an equivalence class in the relation $\sim_{I}$, then the set:
\begin{equation*}
N^{I}_{D}=\{ x \in \partial^{I}\Lambda \ : \ x \text{ eventually lies in }D \}
\end{equation*}
is a Borel set. If $D$ is a $(I,\beta ,r)$-subharmonic component, then the measure $m_{D}$ associated to the corresponding $\beta$-KMS state for $\alpha^{r}$ satisfies $m_{D}(N^{I}_{D})=1$.
\end{lemma}

\begin{proof}
If $I= \emptyset$ then $\partial^{I}\Lambda =\Lambda$, $D=\{v \}$ for some vertex $v \in \Lambda^{0}$ and the set $N^{I}_{D}$ is the countable set of paths $\lambda \in \Lambda$ with $s(\lambda)=v$, hence in particular it is a Borel set. Assume that $I \neq \emptyset$, and set:
\begin{equation} \label{ejohan}
N_{\overline{D}^{I}}^{I}:=\{ x \in  \partial^{I}\Lambda \ : x \text{ eventually  lies in } \overline{D}^{I} \} 
\end{equation} 
Take $x \in  \partial^{I}\Lambda$ with $x \notin N^{I}_{\overline{D}^{I}}$. In particular there is a $m \in \mathbb{N}^{k}$ with $(0, d(x)_{J})\leq m \leq d(x)$ such that $r(\sigma^{m}(x))\notin \overline{D}^{I}$. Setting $\lambda = x(0, m) \in \Lambda^{m}$ and letting $E:=\{e \in s(\lambda)\Lambda \ : \ d(e)=e_{j} \text{ for some } j \notin I\}$, we see that:
\begin{equation*}
\left[ Z(\lambda) \setminus \left(\bigcup_{e \in E} Z(\lambda e) \right) \right] \cap N_{\overline{D}^{I}}^{I} = \emptyset
\end{equation*}
however $x$ is contained in the set we intersect with $N_{\overline{D}^{I}}^{I}$, so $\left[ \Lambda^{*} \setminus N_{\overline{D}^{I}}^{I} \right] \cap \partial^{I}\Lambda$ is an open set in $\partial^{I}\Lambda$, which implies that $N_{\overline{D}^{I}}^{I}$ is a Borel set. Letting $\mathcal{M}$ be the set of equivalence classes in $\sim_{I}$ contained in $\overline{D}^{I} \setminus D$, then:
\begin{equation*}
N_{D}^{I}= N_{\overline{D}^{I}}^{I} \setminus \left( \bigcup_{C \in \mathcal{M}} N_{\overline{C}^{I}}^{I}  \right)
\end{equation*}
and hence $N_{D}^{I}$ is Borel. The sets $N_{D}^{I}$, where $D$ is an equivalence set in $\sim_{I}$, is a disjoint Borel partition of $\partial^{I}\Lambda$, also when $I= \emptyset$. It follows from Theorem \ref{t58} that $m_{D}$ is extremal in the set of quasi-invariant Borel probability measures with Radon-Nikodym cocycle $e^{-\beta c_{r}}$, and hence $m_{D}$ maps invariant Borel sets to $\{0,1\}$. Since $m_{D}(\partial^{I}\Lambda)=1$ there must therefore exist exactly one equivalence class $C$ in $\sim_{I}$ such that $m_{D}(N_{C}^{I})=1$. We know from Lemma \ref{l56} that $m_{D}(v\Lambda^{\infty_{I}, 0})>0$ if and only if $v \in \overline{D}^{I}$. However this must imply that for each $v \in \overline{D}^{I}$ we have $m_{D}(v\Lambda^{\infty_{I}, 0} \cap N_{C}^{I})>0$, which implies that $v \in \overline{C}^{I}$, so that in particular $D \subseteq \overline{C}^{I}$. Considering a $v\in D$, it follows since $m_{D}(v\Lambda^{\infty_{I}, 0} \cap N_{C}^{I})>0$ that there is an $\alpha \in v\Lambda C$ such that $m_{D}(\alpha \Lambda^{\infty_{I}, 0})>0$, and hence using Lemma \ref{l55} we get that $m_{D}(s(\alpha)\Lambda^{\infty_{I}, 0})>0$. This implies that $s(\alpha) \in \overline{D}^{I}$, so we also get $C \supseteq \overline{D}^{I}$, and hence $C=D$.
\end{proof}

To describe the non gauge-invariant KMS states, fix a $\beta \in \mathbb{R}$ and $r\in \mathbb{R}^{k}$, and let $C$ be a $(I, \beta, r)$-subharmonic component for some $I\subseteq \{1, \dots , k\}$. When $I\neq \emptyset$ we define the Periodicity group $\text{Per}_{I}(C)$ as:
\begin{equation*}
\{  (m,0)-(n,0) \ : \ m,n \in \mathbb{N}^{I} \ , \  \sigma^{(m,0)}(x)=\sigma^{(n,0)}(x) \text{ for all } x\in C\Lambda^{\infty_{I}, 0}\cap N_{C}^{I} \}
\end{equation*}
Since $C$ is a component in the $|I|$-graph $\Lambda_{I}$, $C\Lambda_{I}C$ is a strongly connected $|I|$-graph without sources and sinks, and hence it has an infinite path space $(C\Lambda_{I}C)^{\infty}$ consisting of functors from $\Omega_{|I|}$ to $C\Lambda_{I}C$. By identifying $\Omega_{|I|}$ with $\Omega_{k, (\infty_{I}, 0)}$ we get a homeomorphism from $(C\Lambda_{I}C)^{\infty}$ to $C\Lambda^{\infty_{I}, 0}\cap N_{C}^{I}$ that sends the shift map of degree $n\in \mathbb{N}^{I}$ on $(C\Lambda_{I}C)^{\infty}$ to the shift $\sigma^{(n,0)}$ on $C\Lambda^{\infty_{I}, 0}\cap N_{C}^{I}$, and that sends a cylinder set in $(C\Lambda_{I}C)^{\infty}$ given by $\lambda \in C\Lambda_{I}C$ to the relatively open set $Z(\lambda) \cap \left( C\Lambda^{\infty_{I}, 0}\cap N_{C}^{I} \right)$. It follows from this that our periodicity group $\text{Per}_{I}(C)$ is isomorphic to the periodicity group $\text{Per}(C\Lambda_{I}C)$ for the $I$-graph $C\Lambda_{I}C$ introduced in Section 5 in \cite{aHLRS}, and so by Proposition 5.2 in \cite{aHLRS} it is in fact a group. When $I=\emptyset$, we let $\text{Per}_{I}(C)=\{0\}$ with $0\in \mathbb{Z}^{k}$. Using the continuous map $\Phi :\mathcal{G}_{\Lambda} \to \mathbb{Z}^{k}$ defined in section \ref{background} we can now describe the non gauge-invariant KMS states.

\begin{thm}\label{t64}
Let $\Lambda$ be a finite $k$-graph and fix  $r \in \mathbb{R}^{k}$ and $\beta \in \mathbb{R}\setminus\{0\}$. There is a bijection between pairs $(C , \xi)$, where $C \in \mathcal{C}_{r}^{I}(\beta)$ for some $I\subseteq \{1, \dots , k\}$ and $\xi$ lies in the dual $\widehat{\text{Per}_{I}(C)}$ of $\text{Per}_{I}(C)$, to the set of extremal $\beta$-KMS states for $\alpha^{r}$ on $\mathcal{T}C^{*}(\Lambda)$:
\begin{equation*}
(C, \xi) \to \omega_{C, \xi}
\end{equation*}
where:
\begin{equation*}
\omega_{C, \xi}(f) = \int_{\Lambda^{*}} \sum_{g \in \mathcal{G}_{x}^{x}} f(g)\xi(\Phi(g)) \ d m_{C}(x) \quad \text{ for all } f\in C_{c}(\mathcal{G}).
\end{equation*} 
\end{thm}

\begin{remark}
The observation made after Definition \ref{d54} is also true here; the notations $\omega_{C, \xi}$ and $m_{C}$ are only well defined because we have fixed $\beta$ and $r$.
\end{remark}

\begin{remark}
Theorem \ref{t64} gives also a complete description of the $0$-KMS states for $\alpha^{r}$. The $0$-KMS states are the tracial states on $\mathcal{T}C^{*}(\Lambda)$, but choosing $0\in \mathbb{R}^{k}$ this is the same as the $1$-KMS states for $\alpha^{0}$. 
\end{remark}

\begin{proof}[Proof of Theorem \ref{t64}]
Let $C$ be $(I, \beta ,r)$-subharmonic. Let $A$ denote the subgroup of $\mathbb{Z}^{k}$ ensured by Theorem 5.2 in \cite{C} that satisfies:
\begin{equation*}
m_{C}(\{ x\in \Lambda^{*} \ : \ \Phi(\mathcal{G}_{x}^{x})=A  \})=1
\end{equation*}
and denote this Borel set by  $X(A)$. Using Lemma 4.1 and Theorem 5.2 in \cite{C} it is enough to prove that $A=\text{Per}_{I}(C)$ to prove the Theorem. Since $m_{C}(X(A) \cap N_{C}^{I})=1$, we can pick a $ x \in X(A) \cap N_{C}^{I}$, then setting $J=\{1, \dots , k\} \setminus I$ there is a $m \in \mathbb{N}^{k}$ such that $(0,d(x)_{J}) \leq m \leq d(x)$ and  $r(\sigma^{l}(x)) \in C$ for all $m \leq l \leq d(x)$, i.e. $\sigma^{m}(x)\in C\Lambda^{\infty_{I}, 0}\cap N_{C}^{I}$. For $l \in \text{Per}_{I}(C)$ we can write $l=(s,0)-(p,0)$ for some $s,p\in \mathbb{N}^{I}$ such that $\sigma^{(s,0)}(\sigma^{m}(x))=\sigma^{(p,0)}(\sigma^{m}(x))$, so $l\in \Phi(\mathcal{G}^{x}_{x})$. Since $x\in X(A)$ this implies that $\text{Per}_{I}(C) \subseteq A$.

For the other inclusion, let $x\in X(A)\cap N_{C}^{I}$ and assume that there is a $a \in A$ with $a_{j} > 0$ for a $j \in J$. Since $a\in A$ there exists $n,m$ such that $a=n-m$, $n_{j}>m_{j}$ and $\sigma^{n}(x)=\sigma^{m}(x)$. This implies that $d(\sigma^{n}(x))=d(\sigma^{m}(x))$, so $d(x)_{j}=\infty$. Since $j \in J$ and $x \in N^{I}_{C} \subseteq \partial^{I}\Lambda$ this is a contradiction. So $a_{j}=0$ for $j \in J$, which in particular proves the Theorem when $I = \emptyset$. So assume $I \neq \emptyset$ and that there exists $a\in A\setminus \text{Per}_{I}(C)$. Now fix $v \in C$, then:
\begin{equation*}
v\Lambda^{\infty_{I}, 0} \cap X(A)
\subseteq \bigcup_{n,m \in \mathbb{N}^{k}, n-m=a} \{ x\in v\Lambda^{\infty_{I}, 0} \ : \ \sigma^{n}(x)=\sigma^{m}(x)\}
\end{equation*}
so since $m_{C}(v\Lambda^{\infty_{I} , 0}\cap X(A))>0$ we can find a $n_{1}, n_{2} \in \mathbb{N}^{I}$ with $(n_{1}-n_{2},0)=a$ and 
\begin{equation} \label{erty}
m_{C}(\{ x\in v\Lambda^{\infty_{I}, 0} \ : \ \sigma^{(n_{1},0)}(x)=\sigma^{(n_{2},0)}(x)\})>0 .
\end{equation}
For the vector $y^{C}$ corresponding to $m_{C}$ we have $A_{i}y^{C}=e^{\beta r_{i}}y^{C}=\rho(A_{i}^{C})y^{C}$ for $i\in I$, which implies that 
\begin{equation*}
A_{i}^{C}y^{C}|_{C}=e^{\beta r_{i}}y^{C}|_{C}=\rho(A_{i}^{C})y^{C}|_{C}
\end{equation*}
Since $(A_{i}^{C})_{i\in I}$ are the vertex matrices for $C\Lambda_{I}C$, it follows from $(b)$ in Corollary 4.2, Proposition 8.1 and Proposition 8.2 in \cite{aHLRS} that there is a Borel probability measure $M$ on $(C\Lambda_{I}C)^{\infty}\simeq C\Lambda^{\infty_{I}, 0}\cap N_{C}^{I}$, with 
\begin{equation*}
M(\{x\in C\Lambda^{\infty_{I}, 0}\cap N_{C}^{I} \ : \ \sigma^{(n_{1},0)}(x)= \sigma^{(n_{2},0)}(x)  \})=0
\end{equation*}
and that for each $\lambda\in C\Lambda_{I} C$ satisfies:
\begin{equation*}
M(\ C\Lambda^{\infty_{I}, 0}\cap N_{C}^{I} \cap Z(\lambda) \} =e^{-\beta r\cdot d(\lambda)} y^{C}_{s(\lambda)} \lVert y^{C}|_{C}\rVert_{1}^{-1} .
\end{equation*}
Let $\varepsilon >0$. By compactness there are paths $\delta_{1}, \dots , \delta_{q} \in C\Lambda_{I}C$ of the same degree such that:
\begin{equation*}
\{x\in C\Lambda^{\infty_{I}, 0}\cap N_{C}^{I} \ : \ \sigma^{(n_{1},0)}(x)= \sigma^{(n_{2},0)}(x)  \}
 \subseteq \bigsqcup_{i=1}^{q} Z(\delta_{i})\cap C\Lambda^{\infty_{I}, 0}\cap N_{C}^{I}
\end{equation*}
and such that $\sum_{i=1}^{q}M(Z(\delta_{i})\cap C\Lambda^{\infty_{I}, 0}\cap N_{C}^{I})< \varepsilon$. Since $\varepsilon$ was arbitrary the calculation:
\begin{align*}
\sum_{i=1}^{n} m_{C}(Z(\delta_{i}))
&=  \sum_{i=1}^{n} e^{- \beta r \cdot d(\delta_{i})}m_{C}(Z(s(\delta_{i})))
=  \sum_{i=1}^{n} e^{- \beta r \cdot d(\delta_{i})}y_{s(\delta_{i})}^{C}\\
&=\lVert y^{C}|_{C}\rVert_{1}  \sum_{i=1}^{n}M(Z(\delta_{i})\cap C\Lambda^{\infty_{I}, 0}\cap N_{C}^{I})< \varepsilon
\end{align*}
implies that:
\begin{equation*}
m_{C}(\{x\in C\Lambda^{\infty_{I}, 0}\cap N_{C}^{I} \ : \ \sigma^{(n_{1},0)}(x)= \sigma^{(n_{2},0)}(x)  \}) =0 .
\end{equation*}
Combining this with \eqref{erty} gives us a contradiction since any $x\in v\Lambda^{\infty_{I},0}$ with $x\notin N_{C}^{I}$ satisfies $x\in Z(\alpha)$ for some $\alpha$ with $m_{C}(Z(\alpha))=0$. In conclusion $A=\text{Per}_{I}(C)$.
\end{proof}

We will now conclude from Theorem \ref{t64} which KMS states factors through a KMS state of the Cuntz-Krieger algebra $C^{*}(\Lambda)$. To introduce the Cuntz-Krieger algebra we follow \cite{Cynthia, yeend}. Let $\Lambda$ be a finite $k$-graph. For $v\in \Lambda^{0}$ a subset $E \subseteq v\Lambda$ is called a \emph{finite exhaustive set} if it is finite and for every $\lambda \in v\Lambda$ there exists $\mu \in E$ with $\Lambda^{\text{min}}(\lambda, \mu)\neq \emptyset$. For each $v\in \Lambda^{0}$ we let $v\mathcal{F}\mathcal{E}(\Lambda)$ denote the set of finite exhaustive sets $E$ with $E \subseteq v\Lambda$, and we call a path $x\in \Lambda^{*}$ a \emph{boundary path} if for all $n\in \mathbb{N}^{k}$ with $n\leq d(x)$ and all $E\in r(\sigma^{n}(x)) \mathcal{F}\mathcal{E}(\Lambda)$ there exists a $\lambda \in E$ with $x(n,n+d(\lambda))=\lambda$. By Proposition 6.10 in \cite{Cynthia}, the set $\mathcal{B}$ of boundary paths in $\Lambda^{*}$ is a non-empty invariant closed set, so the reduction $\mathcal{G}_{\Lambda}|_{\mathcal{B}}$ of $\mathcal{G}_{\Lambda}$ to $\mathcal{B}$ is a locally compact second countable Hausdorff \'etale groupoid, and its full groupoid $C^{*}$-algebra $C^{*}(\mathcal{G}_{\Lambda}|_{\mathcal{B}})$ is the Cuntz-Krieger algebra of $C^{*}(\Lambda)$. To conclude which KMS states factor through a state of $C^{*}(\Lambda)$ we need the following:

\begin{lemma} \label{lslut}
Let $\Lambda$ be a finite $k$-graph and $I\sqcup J=\{1, \dots , k\}$ be a partition. A path $x\in \partial^{I}\Lambda$ is a boundary path if and only if $x\in N_{\overline{C}^{I}}^{I}$ for a component $C$ in $\sim_{I}$ satisfying:
\begin{equation} \label{eslut}
C\Lambda^{e_{j}} = \emptyset \qquad \forall j\in J
\end{equation}
where $N_{\overline{C}^{I}}^{I}$ is defined in \eqref{ejohan}.
\end{lemma}

\begin{proof}
Assume $x\in \partial^{I}\Lambda$ is a boundary path. Since $x\in N_{\overline{C}^{I}}^{I}$ is equivalent to $\sigma^{n}(x)\in N_{\overline{C}^{I}}^{I}$ for any $n\in \mathbb{N}^{k}$ with $n\leq d(x)$, we can assume by Lemma 5.13 in \cite{Cynthia} that $x\in \Lambda^{\infty_{I}, 0}$ and that there exists a component $D$ in $\sim_{I}$ with $r(\sigma^{n}(x))\in D$ for all $n\leq d(x)$. Set $v=r(x)$ and assume for contradiction that $x\notin N_{\overline{C}^{I}}^{I}$ for any $C$ satisfying \eqref{eslut} and set:
\begin{equation*}
E := \bigcup_{j\in J} v\Lambda^{e_{j}}  .
\end{equation*}
For $\lambda \in v\Lambda$ with $d(\lambda)_{j} >0$ for some $j\in J$ it follows from the factorisation property that there exists a $e\in E$ with $\Lambda^{\text{min}}(e,\lambda)\neq \emptyset$. If $d(\lambda)\in \mathbb{N}^{I}$ then there is some component $C$ in $\sim_{I}$ with $s(\lambda)\in C$. This implies that $D\subseteq \overline{C}^{I}$, so by assumption $C\Lambda^{e_{j}}\neq \emptyset$ for some $j\in J$. Since $C$ is a component there is a $\mu\in s(\lambda)\Lambda$ with $d(\mu)_{j}>0$, and using the factorisation property on $\mu$ this implies $s(\lambda)\Lambda^{e_{j}} \neq \emptyset$. Composing $\lambda$ with an element from this set, and using the factorisation property, gives a $e\in E$ with $\Lambda^{\text{min}}(e, \lambda)\neq \emptyset$. In conclusion $E\in v\mathcal{F}\mathcal{E}(\Lambda)$, contradicting that $x \in \Lambda^{\infty_{I}, 0}$ is a boundary path. This proves one implication.

Assume now that $x\in N_{\overline{C}^{I}}^{I}$ for a component $C$ in $\sim_{I}$ with $C\Lambda^{e_{j}}=\emptyset$ for all $j\in J$. We can again assume that $x\in \Lambda^{\infty_{I}, 0}$ and that there exists a component $D$ in $\sim_{I}$ with $r(\sigma^{n}(x))\in D$ for all $n\leq d(x)$. Since $x\in N_{\overline{C}^{I}}^{I}$, this implies that $D \subseteq \overline{C}^{I}$. Now let $n\in \mathbb{N}^{k}$ with $n\leq d(x)$ and $E\in r(\sigma^{n}(x))\mathcal{F}\mathcal{E}(\Lambda)$. Choose $m\in \mathbb{N}^{I}$ such that $m \geq d(\lambda)_{I}$ for all $\lambda \in E$, since $r(\sigma^{n+m}(x))\in D\subseteq \overline{C}^{I}$ there is a path $\mu \in r(\sigma^{n+m}(x))\Lambda_{I} C$. Now $x(n, n+m)\mu \in r(\sigma^{n}(x))\Lambda$ so there is a $e\in E$ with $\Lambda^{\text{min}}(x(n, n+m)\mu, e)\neq \emptyset$, but since $s(x(n, n+m)\mu)=s(\mu)\in C$ we must have $d(e)\in \mathbb{N}^{I}$, and by choice of $m$ we have $x(n, n+d(e))=e$, proving the other implication.
\end{proof}

\begin{remark}
When $J=\emptyset$ all components $C$ in $\sim$ satisfies \eqref{eslut}, so $\Lambda^{\infty} \subseteq \mathcal{B}$. When $I=\emptyset$ we see that $x\in \Lambda$ satisfies $x\in \mathcal{B}$ exactly when $s(x)$ is an \emph{absolute source}, i.e. $s(x)\Lambda=\{s(x)\}$.
\end{remark}

\begin{cor}[Corollary to Theorem \ref{t64}]
In the setting of Theorem \ref{t64} a state $\omega_{D, \xi}$ for a $D\in \mathcal{C}_{r}^{I}(\beta)$ and $\xi\in \widehat{\text{Per}_{I}(D)}$ factors through a state of $C^{*}(\Lambda)$ if and only if $D\subseteq \overline{C}^{I}$ for a component $C$ in $\sim_{I}$ satisfying:
\begin{equation} \label{eslut2}
C\Lambda^{e_{j}}=\emptyset \qquad \forall j\in \{1, \dots , k\} \setminus I
\end{equation}
\end{cor}

\begin{proof}
Since $\mathcal{B}$ is invariant we either have $m_{D}(\mathcal{B})=1$ or $m_{D}(\mathcal{B})=0$, and so Lemma 2.10 in \cite{MRW} implies that $\omega_{D, \xi}$ factors through a state of $C^{*}(\Lambda)$ precisely when $m_{D}(\mathcal{B})=1$. If $m_{D}(\mathcal{B})=1$ then $m_{D}(\mathcal{B}\cap N_{D}^{I})=1$ and there must be a $x\in \partial^{I}\Lambda\cap \mathcal{B}$ that eventually lies in $D$. Lemma \ref{lslut} now implies that $x$ also eventually lies in $\overline{C}^{I}$ for a $C$ satisfying \eqref{eslut2}, so we must have $D\subseteq \overline{C}^{I}$. If $m_{D}(\mathcal{B})=0$ and there exists a $C$ as in \eqref{eslut2} with $D\subseteq \overline{C}^{I}$ then $N_{D}^{I}\subseteq N_{\overline{C}^{I}}^{I}$, but this is not possible since $m_{D}(N_{D}^{I})=1$ and $m_{D}(N_{\overline{C}^{I}}^{I}) \leq m_{D}(\mathcal{B})=0$. This proves the corollary.
\end{proof}

\section{Examples and comparison with the literature} \label{ex}

\subsection{Examples}
To illustrate how to describe the KMS states for a given graph we will use our machinery on a few examples. The first graph we consider is from Example 9.1 in \cite{FaHR} where the KMS states for the action given by $r=(\ln(5), \ln(4))$ were calculated. We have included this example to illustrate that our results give the same KMS states as the ones in \cite{FaHR}, but also to show the strength of our approach when it comes to concrete calculations.

\begin{example}
Consider a $2$-graph given by the graph below, where normal edges have degree $e_{1}$ and dashed edges have degree $e_{2}$, and the number at each edge denote the number of edges:
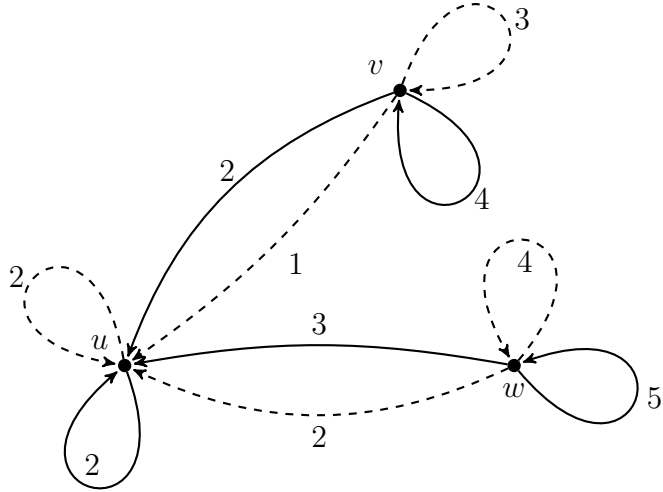
\begin{center}
\tikzstyle{input} = [draw, circle, fill, inner sep=1.5pt]
\begin{tikzpicture}[->,>=stealth',shorten >=1pt,auto,node distance=2cm,
                    thick,main node/.style={circle,draw,font=\sffamily\small\bfseries} ,align=center]

  \node[input](1) [label=above left:$u$]{} ;
  \node[input] (2) [above right=5cm of 1, label=above left:$v$] {};
    \node[input] (3) [right=5cm of 1, label=below:$w$] {};

  \path[every node/.style={font=\sffamily}]
    (2)  edge [dashed, bend left, out=10, in=170] node {$1$} (1)
       (2)  edge [bend right, out=-25, in=205] node[above] {$2$} (1)
     (3)  edge [dashed, bend left, out=25, in=155] node {$2$} (1)      
           (3)  edge [bend right, out=-10, in=190] node[above] {$3$} (1)    
    (1)  edge [dashed, loop, min distance=26mm,in=170,out=100,looseness=10] node[left] {$2$} (1)
        (1)  edge [loop, min distance=26mm,in=-140,out=-70,looseness=10] node[above] {$2$} (1)
        (2)  edge [dashed, loop, min distance=26mm,in=0,out=70,looseness=10] node[right] {$3$} (2)
                (2)  edge [loop, min distance=26mm,in=-95,out=-25,looseness=10] node[right] {$4$} (2)
                (3)  edge [loop, min distance=26mm,in=20,out=-50,looseness=10] node[right] {$5$} (3)
                (3)  edge [dashed, loop, min distance=26mm,in=120,out=50,looseness=10] node[below] {$4$} (3);
\end{tikzpicture}
\end{center}
No matter which $I\subseteq \{1,2\}$ we choose there are three components $\{u\}$, $\{v\}$ and $\{w\}$ for $\sim_{I}$, so we will analyse for which $\beta$ and $r$ each is $(I, \beta ,r)$-subharmonic. For this, notice that $\overline{\{v\}}=\{u,v\}$, $\overline{\{u\}}=\{u\}$ and $\overline{\{w\}}=\{u,w\}$. Considering the graph it follows that:
\begin{align*}
&\rho(A^{\overline{\{v\}}}_{1})=\rho(A^{\{v\}}_{1})=4 \quad , \quad \rho(A^{\overline{\{v\}}}_{2})=\rho(A^{\{v\}}_{2})=3 \\
& \rho(A^{\overline{\{u\}}}_{1})=\rho(A^{\{u\}}_{1})=2 \quad , \quad \rho(A^{\overline{\{u\}}}_{2})=\rho(A^{\{u\}}_{2})=2 \\
& \rho(A^{\overline{\{w\}}}_{1})=\rho(A^{\{w\}}_{1})=5 \quad , \quad \rho(A^{\overline{\{w\}}}_{2})=\rho(A^{\{w\}}_{2})=4
\end{align*}
Hence by Definition 5.1 the different components give KMS states for $\beta r$ in the sets as indicated in the table below. 

\begin{center}
  \begin{tabular}{  l | c | c | c |}
    
    I$\diagdown$C & $\{v\}$ & $\{u\}$ & $\{w\}$ \\ 
    \hline 
    $\emptyset$ & $]\ln(4), \infty[\times ]\ln(3), \infty[$ & $]\ln(2), \infty[\times ]\ln(2), \infty[$ & $]\ln(5), \infty[\times ]\ln(4), \infty[$ \\ \hline
    $\{1\}$ & $\{\ln(4)\} \times ]\ln(3), \infty[$ & $\{\ln(2)\}\times ]\ln(2), \infty[$ & $\{\ln(5)\}\times ]\ln(4), \infty[$  \\ \hline
    $\{2\}$ & $]\ln(4), \infty[\times \{\ln(3)\}$ & $]\ln(2), \infty[\times \{\ln(2)\}$ & $]\ln(5), \infty[\times \{\ln(4)\}$ \\ \hline
     $\{1,2\}$ & $\{\ln(4)\}\times \{\ln(3)\}$ & $\{\ln(2)\}\times \{\ln(2)\}$ & $\{\ln(5)\}\times \{\ln(4)\}$ \\
    \hline
  \end{tabular}
\end{center}
As in \cite{FaHR} we now consider the action given by $r=(\ln(5), \ln(4))$ which has rationally independent coordinates. Theorem \ref{t58} then implies that we get a complete description of the $\beta$-KMS states for $\alpha^{r}$ by describing the $(I, \beta, r)$-subharmonic components for different $I\subseteq \{1, 2\}$. So we go through each entry of the table and consider for which value of $\beta$ that $\beta r$ lies in the set at that entry. This gives the following result (notice $\ln(2)/\ln(4)=1/2$):
\begin{center}
  \begin{tabular}{  l | c | c | c |}
    
    I$\diagdown$C & $\{v\}$ & $\{u\}$ & $\{w\}$ \\ 
    \hline 
    $\emptyset$ & $]\ln(4)/\ln(5), \infty[$ & $]1/2, \infty[$ & $]1, \infty[$ \\ \hline
    $\{1\}$ & $\{\ln(4)/\ln(5)\}$ & $\emptyset$ & $\emptyset$  \\ \hline
    $\{2\}$ & $\emptyset$ & $\{1/2\}$ & $\emptyset$ \\ \hline
     $\{1,2\}$ & $\emptyset$ & $\emptyset$ & $\{1\}$ \\
    \hline
  \end{tabular}
\end{center}
This is exactly the same as obtained in Example 9.1 in \cite{FaHR}.
\end{example}

\begin{example}
Using Theorem \ref{t58} we will give an example of a strongly connected graph without sources and sinks and a one-parameter group $\alpha^{r}$ with two different gauge-invariant $\beta$-KMS states for $\alpha^{r}$ for the critical temperature $\beta$. To do this consider the following skeleton:
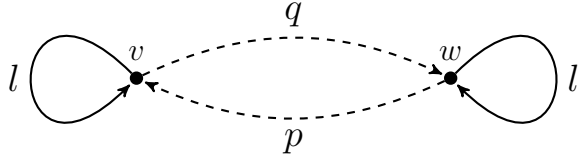
\begin{center}
\tikzstyle{input} = [draw, circle, fill, inner sep=1.5pt]
\begin{tikzpicture}[->,>=stealth',shorten >=1pt,auto,node distance=2cm,
                    thick,main node/.style={circle,draw,font=\sffamily\small\bfseries} ,align=center, scale=0.7]

  \node[input](1) [label=above:$v$]{} ;
  \node[input] (2) [right=4cm of 1, label=above:$w$] {};

  \path[every node/.style={font=\sffamily\large}]
    (2)  edge [dashed,bend left, out=25, in=155] node {$p$} (1)
    (1)  edge [dashed, bend right, out=25, in=155]node {$q$} (2)
    (1)  edge [loop, min distance=36mm,in=225,out=135,looseness=10] node[left] {$l$} (1)
    (2)  edge [ loop, min distance=36mm,out=45, in=-45, looseness=10 ] node {$l$} (2);
\end{tikzpicture}
\end{center}
The full edges have degree $e_{1}$ and the dashed edges have degree $e_{2}$, and the numbers $l,p,q \geq 1$ denote the number of edges . Since $A_{1}=l 1_{\{v,w\}}$ then $A_{1}$ and $A_{2}$ must commute, and hence there exists a $2$-graph with this skeleton, c.f. Section 6 in \cite{KP}. In the equivalence relation $\sim_{\{1\}}$ both $\{v\}$ and $\{w\}$ are components, and choosing $r=(\ln(l),\ln(2\sqrt{p\cdot q}))$ we see that $e^{1 r_{1} }=\rho(A^{\{v\}}_{1})=\rho(A^{\{w\}}_{1})$. Since $\overline{\{v\}}=\overline{\{w\}}=\{v,w\}$, and since $\rho(A_{2})=\sqrt{p\cdot q}$, both $\{v\}$ and $\{w\}$ are $(\{1\}, 1,r)$-subharmonic components, and since there are no $(\emptyset, 1,r)$- and $(\{1,2\}, 1,r)$-subharmonic components, Theorem \ref{t58} implies that they give rise to the only extremal gauge-invariant $1$ -KMS states for $\alpha^{r}$. Ordering the set of vertexes by $\{v,w\}$ then the vectors given in Proposition \ref{p53} are:
\begin{equation*}
\tilde{x}^{\{v\}} = \frac{2}{3}
\begin{pmatrix}
&2 \\
&\sqrt{q}/ \sqrt{p}
\end{pmatrix}
\quad , \quad 
\tilde{x}^{\{w\}} = \frac{2}{3}
\begin{pmatrix}
&\sqrt{p}/ \sqrt{q} \\
&2
\end{pmatrix}.
\end{equation*}
Both vectors are, as predicted, sub-invariant for the family $\{l^{-1} A_{1}, (2 \sqrt{p\cdot q})^{-1} A_{2}\}$. Hence their normalizations $y^{\{v\}}$ and $y^{\{w\}}$ give rise to two different gauge-invariant $1$-KMS states for $\alpha^{r}$. When $l>1$ then $\text{Per}_{\{1\}}(\{v\})=\text{Per}_{\{1\}}(\{w\})=\{0\}$, and the $1$-KMS states for $\alpha^{r}$ are given by convex combinations of the two states $\omega_{y^{\{v\}}}$ and $\omega_{y^{\{w\}}}$, with:
\begin{equation*}
\omega_{y^{\{v\}}}(S_{\lambda}S_{\mu}^{*})=\delta_{\lambda, \mu} e^{-\beta r\cdot d(\lambda)}y^{\{v\}} \quad , \quad  \omega_{y^{\{w\}}}(S_{\lambda}S_{\mu}^{*})=\delta_{\lambda, \mu} e^{-\beta r\cdot d(\lambda)}y^{\{w\}} .
\end{equation*}
If $l=1$ then $\text{Per}_{\{1\}}(\{v\})=\text{Per}_{\{1\}}(\{w\})=\mathbb{Z} \times \{0\}$, so letting $m_{v}$  and $m_{w}$ be the measures corresponding to respectively $\{v\}$ and $\{w\}$, then $\Phi(\mathcal{G}_{x}^{x})=\mathbb{Z} \times \{0\}$ for almost all $x\in \Lambda^{*}$ and the extremal $1$-KMS states for $\alpha^{r}$ are:
\begin{equation*}
\omega_{\{u\}, \lambda}(f) = \int_{\Lambda^{*}} \sum_{(x,(n,m), x)\in \mathcal{G}_{x}^{x}} f(x,(n,m), x) \lambda^{n} \ dm_{u}(x) 
\end{equation*}
for all $\lambda \in \mathbb{T}$ and $u=v,w$.
\end{example}

\subsection{Comparison with the literature}
We will now compare our results to the ones in \cite{FaHR}. To do this we will need the following Lemma:

\begin{lemma}\label{lap}
Let $\Lambda$ be a finite $k$-graph without sources with the property, that when $v,w\in \Lambda^{0}$ satisfies $v\Lambda^{m} w\neq \emptyset$ for some $m\in \mathbb{N}^{k}\setminus \{0\}$, then they also satisfy $v\Lambda^{m} w\neq \emptyset$ for some $m\in \mathbb{N}^{\{i\}}\setminus \{0\}$ for each $i\in \{1, \dots , k\}$. Assume $C$ is a component in $\sim$ with $\rho(A_{j}^{C})>0$ for each $j\in \{1, \dots , k\}$, then:
\begin{enumerate}
\item \label{l731}If $\rho(A^{D})\lneq \rho(A^{C})$ for each $D\subseteq \overline{C} \setminus C$, then for all $i\in \{1, \dots , k\}$:
\begin{equation} \label{eap}
\rho(A_{i}^{D}) < \rho(A_{i}^{C}) \quad \text{ for all components } D  \subseteq \overline{C} \setminus C .
\end{equation} 
\item \label{l732} If $C$ satisfies \eqref{eap} for some $i$ then it satisfies it for all $i\in \{1, \dots , k\}$. 
\end{enumerate}
\end{lemma}

\begin{proof}
Fix an $i$. The condition on the graph implies that we can find a finite set of numbers $F\subseteq \mathbb{N}^{\{i\}}\setminus \{0\}$ such that $A_{F}(v,w):=\sum_{n\in F} A^{n}(v,w)> 0$ if and only if $v\Lambda^{m} w \neq \emptyset$ for some $m\in \mathbb{N}^{k}$. Such a set is called well-chosen in the terminology of \cite{C}, and it follows from Lemma 7.11 and Definition 7.5 in that article that $\rho(A_{F}^{\overline{C}\setminus C}) <\rho(A_{F}^{C})$. Using Lemma 7.8 in \cite{C} on the graph $\Lambda\overline{C}\setminus C$ we get $\rho(A_{F}^{D})<\rho(A_{F}^{C}) $ for each $D \subseteq \overline{C}\setminus C$. Since $\rho(A_{F}^{D})=\sum_{n\in F} \rho(A_{i}^{D})^{n_{i}}$ by equation $(7.2)$ in \cite{C} \eqref{eap} follows.

To prove \ref{l732} assume $C$ satisfies \eqref{eap} for a $i$ and choose $F$ as above for this $i$. Equation $(7.2)$ in \cite{C} gives $\rho(A_{F}^{D})<\rho(A_{F}^{C})$ for all components $D\subseteq \overline{C}\setminus C$, and hence combining Lemma 7.8, Definition 7.5 and Lemma 7.11 in \cite{C} imply that $C$ satisfies the criterion in \ref{l731}.
\end{proof}

To follow the set-up in \cite{FaHR} we consider a finite $k$-graph $\Lambda$ and a $r\in \mathbb{R}^{k}$ satisfying condition $1-5$ from the introduction, and we assume that $K\sqcup L=\{1, \dots, k\}$ is a partition with $r_{i}=\ln(\rho(A_{i}))$ for $i\in K\neq \emptyset$ and $r_{l} > \ln(\rho(A_{l}))$ for $l\in L$. We let $\mathcal{C}_{crit}$ be the components $C$ in $\sim$ with $\ln(\rho(A^{C}_{j}))=r_{j}$ for some $j$, and $\mathcal{C}_{mincrit}$ be the minimal elements in $\mathcal{C}_{crit}$ for the order $\leq$. Notice that condition $4$ and $5$ imply that the condition imposed on $\Lambda$ in Lemma \ref{lap} is satisfied, and that the relations $\leq_{I}$, $I\neq \emptyset$ are all equal.

Let $C\in \mathcal{C}_{mincrit}$ and set $I=\{i: \ln(\rho(A_{i}^{C}))=r_{i} \}$ then $I\neq \emptyset$. For $i\in I$ $C$ satisfies \ref{l732} in Lemma \ref{lap}, so $\rho(A^{D})_{I} \lneq \rho(A^{C})_{I}$ for all $D\subseteq \overline{C} \setminus C$, and by \eqref{eap} then $\rho(A_{j}^{\overline{C}})=\rho(A_{j}^{C}) <e^{r_{j}}$ for $j\notin I$, so $C$ is $(I, 1, r)$-subharmonic. Assume on the other hand that $C$ is a $(I, 1, r)$-subharmonic component for a $I\neq \emptyset$, then for $i\in I$ we have $\ln(\rho(A_{i}^{C}))=r_{i}$, so $C \in \mathcal{C}_{crit}$. If $C\notin \mathcal{C}_{mincrit}$ then there is a $(J, 1, r)$-subharmonic component $D$ with $D\subseteq \overline{C}\setminus C$ and $J \subseteq \{1, \dots, k\}$ not empty. If $l \notin I$ then:
\begin{equation*}
\rho(A_{l}^{D}) \leq \rho(A_{l}^{\overline{C}}) <e^{r_{l}}
\end{equation*}
so then $l \notin J$, giving $J\subseteq I$. Since $\Lambda_{I}$ as an $I$-graph satisfies the criterion of Lemma \ref{lap}, $C$ satisfies the criterion in \ref{l731} for the graph $\Lambda_{I}$ and $D\subseteq \overline{C}^{I}\setminus C$ we get that $\rho(A_{i}^{D})< \rho(A_{i}^{C})$ for all $i\in I$. For $i\in J$ this implies $e^{r_{i}}=\rho(A_{i}^{D})<\rho(A_{i}^{C}) =e^{r_{i}}$, a contradiction. So $\mathcal{C}_{mincrit}$ is the set of $(I, 1, r)$-subharmonic components with $I \neq \emptyset$.

If $\{v\}$ is a $(\emptyset, 1, r)$-subharmonic component, then $\rho(A_{j}^{\overline{\{v\}}})<e^{r_{j}}$ for each $j\in \{1, \dots , k\}$, so $\overline{\{v\}}$ contains no components from $\mathcal{C}_{crit}$, and hence $v\notin \widehat{\mathcal{C}_{crit}}$. If on the other hand $v\notin \widehat{\mathcal{C}_{crit}}$ then $\overline{\{v\}}$ contains no critical components, so $\rho(A_{j}^{\overline{\{v\}}})<e^{r_{j}}$ for each $j\in \{1, \dots , k\}$, implying that $\{v\}$ is $(\emptyset, 1, r)$-subharmonic. 

Comparing Theorem \ref{t58} with Section 7 in \cite{FaHR} we now see that the vertexes and components giving rise to extremal $1$-KMS states for $\alpha^{r}$ are the same in the two expositions. To see that the states also agree it suffices to argue that the corresponding vectors over $\Lambda^{0}$ agree. For $v\notin \widehat{\mathcal{C}_{crit}}$ this follows from comparing the vector defined in Proposition \ref{p53} when considerind $\{v\}$ a $(\emptyset, 1, r)$-subharmonic component with the one constructed in Theorem 6.1 in \cite{aHLRS1}. For $C\in \mathcal{C}_{mincrit}$ define $H\subseteq \Lambda^{0}$ as in Proposition 3.4 in \cite{FaHR}, then the vector $z$ of unit $1$-norm constructed in \cite{FaHR} corresponding to $C$ satisfies $z_{v}=0$ for $v\notin \overline{C}$ and for all $i\in \{1, \dots , k\}$ that:
\begin{equation*}
A_{i}^{\Lambda^{0} \setminus H}z|_{\Lambda^{0} \setminus H}=\rho(A_{i}^{C}) z|_{\Lambda^{0} \setminus H} .
\end{equation*}
For $I:=\{i : r_{i}=\ln(\rho(A_{i}^{C}))\}$ we get $A_{i}z=e^{r_{i}}z$ for $i\in I$, so $z$ is the unique vector $x^{C}$ from Lemma \ref{l52}. Since $\tilde{x}^{C}$ is supported on $\overline{C}=\overline{C}^{I}$ and satisfies $A_{i}\tilde{x}^{C}=e^{r_{i}} \tilde{x}^{C}$, it has to be a scalar of $x^{C}$, so $y^{C}=z$. This proves that the states obtained in \cite{FaHR} are the same as the ones obtained in Theorem \ref{t58}.

\begin{ack}
The author thanks Klaus Thomsen for supervision and discussion. This work was supported by the DFF-Research Project 2 `Automorphisms and Invariants of Operator Algebras', no. 7014-00145B.
\end{ack}

\bigskip

E-mail address: \textit{johannes@math.au.dk}

\end{document}